\documentclass[reqno,12pt,a4paper]{amsart}
\usepackage{amscd,amsfonts,amssymb,amsthm,latexsym}
\usepackage{mathtools}
\usepackage{bm}
\usepackage{srcltx}

\theoremstyle{plain}
\newtheorem{theorem}{Theorem}[section]
\newtheorem*{theorem*}{Theorem}

\newtheorem{lemma}{Lemma}[section]
\newtheorem{proposition}{Proposition}[section]

\theoremstyle{definition}

\newtheorem*{definition*}{Definition}

\theoremstyle{remark}

\newtheorem*{remark*}{Remark}

\numberwithin{equation}{section}

\begin{document}
\raggedbottom 

\title[On sums of two squares and a basis of order $2$]{On sums of two squares and a basis of order $2$}

\author{Artyom Radomskii}

\begin{abstract} Let $\mathcal{R}$ denote the set of integers $n$ that can be represented as the sum $n = x^2 + y^2$ with $(x,y) = 1$. Let $a$ and $b$ be integers with $a>0$, $a \nmid b$. We show that for sufficiently large positive integer $N$ there are two strings of consecutive positive integers $I_{1}=\{n_1-m,\ldots, n_1+m\}$ and $I_{2}=\{n_2-m, \ldots, n_2+m\}$ such that $m = [(\log N) (\log \log N)^{1/325565}]$, $I_{1}\cup I_{2} \subset [1, N]$, $N = n_1 + n_2$,  and for any $n\in I_{1}\cup I_{2}$ at least one of  $n$ or $an+b$ does not lie in $\mathcal{R}$. In particular, we have $n(an+b)\notin \mathcal{R}$ for all $n\in I_{1}\cup I_{2}$.
\end{abstract}

 \address{HSE University, Moscow, Russian Federation}


\email{artyom.radomskii@mail.ru}

\maketitle

\section{Introduction}

Let $\mathcal{R}$ denote the set of integers $n$ that can be represented as the sum $n = x^2 + y^2$ with $(x,y) = 1$.

Set
\[
C(\rho):= \sup\Big\{\delta\in \Big(0,\frac{1}{2}\Big) : \frac{6\cdot 10^{2\delta}}{\log (1/ (2\delta))}<\rho\Big\}.
\]

Our main result is the following

\begin{theorem}\label{T1}
 Let $a$ and $b$ be integers, $a>0$, $a \nmid b$. Let $0< \delta < C(1/2)$. Then for sufficiently large positive integer $N$ there are two strings of consecutive positive integers $I_{1}=\{n_1-m,\ldots, n_1+m\}$ and $I_{2}=\{n_2-m, \ldots, n_2+m\}$ such that $m = [(\log N) (\log \log N)^{\delta}]$, $I_{1}\cup I_{2} \subset [1, N]$, $N = n_1 + n_2$,  and for any $n\in I_{1}\cup I_{2}$ at least one of  $n$ or $an+b$ does not lie in $\mathcal{R}$. In particular, we have $n(an+b)\notin \mathcal{R}$ for all $n\in I_{1}\cup I_{2}$.
\end{theorem}
We note that numerical calculations show that $C(1/2)>1/325565$, and so we can take $\delta=1/325565$ in Theorem \ref{T1}.

 Recall that a set $A \subseteq \mathbb{N}$ is called a basis of order $k$ if every sufficiently large positive integer can be represented as a sum of $k$ summands from $A$. Theorem \ref{T1} implies that the set
\begin{align*}
\{n\geq 3: &\text{ for any }l\in [n-m(n), n+m(n)] \text{ at least one of $l$ or $al+b$}\\
 &\text{does not lie in $\mathcal{R}$}, \text{ where } m(n)=[(\log n) (\log\log n)^{\delta}]\}
\end{align*}is a basis of order $2$ for any $\delta < C(1/2)$.

In \cite{FKMPT}, Ford, Konyagin, Maynard, Pomerance, and Tao showed that for any non-constant polynomial $f: \mathbb{Z}\to \mathbb{Z}$ with positive leading coefficient, the set $\{n\leq N: f(n)\text{ composite}\}$ contains an interval of consecutive integers of length $\geq \log N (\log \log N)^{\delta}$ for sufficiently large $N$ and for any $0< \delta < C(1/d)$, where $d=\deg f$. Ford and Gabdullin \cite{Ford.Gabdullin} improved the result in \cite{FKMPT}, replacing $C(1/d)$ by $C(1)$. Here we introduce a modification of the constructions from \cite{FKMPT} and \cite{Ford.Gabdullin}.

Let $\mathcal{A}=\{a_{1}<a_{2}<\ldots\}$ be a set of positive integers, and let
\[
G_{\mathcal{A}}(N)=\max_{a_{n}\leq N} (a_{n+1}-a_{n})
\]denote the maximal gap between elements of $\mathcal{A}$ of size at most $N$. Estimates of the quantities $G_{\mathcal{R}}(N)$ and $G_{\mathcal{S}}(N)$ are of interest, where $\mathcal{S}$ is the set of integers representable as the sum of two squares. The best lower bounds at the moment for $G_{\mathcal{R}}(N)$ and $G_{\mathcal{S}}(N)$ are
\[
G_{\mathcal{R}}(N)\gg \log N,\qquad G_{\mathcal{S}}(N)\gg \log N,
\]which are due to Richards \cite{Richards}.

We also note that the technique proposed by Ford, Konyagin, Maynard, Pomerance, and Tao in \cite{FKMPT} allows to find large gaps between consecutive primes. Theorem 1 in \cite{FKMPT} implies that there is a gap between consecutive primes in $[1, N]$ of size
 \[
 \gg\log N (\log\log N)^{\delta}
 \] for any $\delta < C(1)\approx 1/835$. This is stronger than the trivial bound of $(1+o(1))\log N$, which is immediate from the Prime Number Theorem, but is worse than the current best bounds for this problem. Indeed, the problem of finding large gaps between consecutive primes has a long history, and it is currently known that gaps of size
 \[
 \gg \log N\,\frac{\log\log N\log\log\log\log N}{\log\log\log N}
 \]exist below $N$ if $N$ is large enough, a recent result of Ford, Green, Konyagin, Maynard, and Tao \cite{FGKMT.LARGE}. The key interest is that the technique from \cite{FKMPT} applies to much more general sieving situations, to which it appears difficult to adapt the previous techniques.

\section{Notation}

We use $X\ll Y$, $Y\gg X$, or $X=O(Y)$ to denote the estimate $|X|\leq C Y$ for some constant $C>0$, and write $X\asymp Y$ for $X \ll Y \ll X$. Throughout the remainder of the paper, all implied constants in $O$, $\ll$ or $\gg$ may depend on quantities $a$, $b$, $\delta$, and on the positive parameters $K$, $\xi$, $M$, and $\varepsilon$ which we will describe below. We also assume that the quantity $x$ is sufficiently large in terms of all of these parameters.

The notation $X=o(Y)$ means that $\lim_{x\to \infty} X/Y = 0$, and the notation $X\sim Y$  means that $\lim_{x\to \infty} X/Y = 1$.

If $S$ is a statement, we use $1_{S}$ to denote its indicator, thus $1_{S}=1$ when $S$ is true and $1_{S}=0$ when $S$ is false.

We will rely on probabilistic methods in this paper. Boldface symbols such as $\mathbf{d}$, $\mathbf{n}$, $\mathbf{S}$, $\bm{\lambda}$, etc. denote random variables (which may be real numbers, random sets, random functions, etc.) Most of these random variables will be discrete (in fact they will only take on finitely many values), so that we may ignore any technical issues of measurability. We use $\mathbb{P}(\mathbf{E})$ to denote the probability of a random event $\mathbf{E}$, and $\mathbb{E}(\mathbf{X})$ to denote the expectation of the random (real-valued) variable $\mathbf{X}$.

The symbol $p$ (as well as variants such as $p_1$, $p_2$, etc.) will always denote a prime. We write $\nu(n)$ for the number of distinct prime divisors of $n$, $\tau(n)$ for the number of positive integer divisors of $n$, and $\varphi(n)$ for Euler's function (the order of the multiplicative group of reduced residue classes modulo $n$).

If $x$ is a real number, then $[x]$ denotes its integral part.

By $\# A$ we denote the number of elements of a finite set $A$. Let $(a,b)$ denote the greatest common divisor of integers $a$ and $b$.

To ease notation, we let $<P_{1}>$ denote the set of integers composed only of primes congruent to $1$ (mod $4$). Similarly let $<P_{3}>$ those composed of primes congruent to $3$ (mod $4$).

\section{Setup}\label{SECTION.SETUP}

It is well-known that $n\in \mathcal{R}$ if and only if $n= 2^{\alpha} m$, where $\alpha\in \{0,1\}$ and $m\in <P_{1}>$ (see, for example, \cite[Ch. VI]{Davenport}).

 It is clear that $a \nmid b$ requires that $b \neq 0$. Also, we observe that $\delta \in (0,1/2)$ satisfies $\delta< C(\rho)$ if and only if
\begin{equation}\label{C.RHO.INEQUALITY}
\frac{6\cdot 10^{2\delta}}{\log (1/ (2\delta))}<\rho.
\end{equation}

For a fixed $\delta\in (10^{-6}, C(1/2))$, we define
\begin{equation}\label{SETUP:Def.y}
y= [x(\log x)^{\delta}]\qquad\text{and}\qquad z=\frac{y\log\log x}{(\log x)^{1/2}}.
\end{equation} It is clear that $\log y \sim \log x$ and $\log z \sim \log x$.

By $q$ we will always denote a prime such that $q> a+|b|$ and $q\equiv 3$ (mod $4$). Let $c_{q}$ be a solution of the congruence $a c_{q} + b \equiv 0 $ (mod $q$). We note that this solution exists, unique and $c_{q}\not \equiv 0$ (mod $q$), since $(a,q)=1$ and $(b,q)=1$. We assume that $c_{q}\in [1,q-1]$. Put
\[
I_{q}= \{0, c_{q}\}.
\]

We define
\begin{gather*}
P(x):= \prod_{q\leq x}q,\quad\qquad \sigma (x):= \prod_{q\leq x}\bigg(1-\frac{2}{q}\bigg),\\
 P (z,x):=\prod_{z<q\leq x}q,\quad\qquad \sigma (z,x):= \prod_{z<q\leq x}\bigg(1-\frac{2}{q}\bigg).
\end{gather*} By the Prime Number Theorem, we have
\begin{equation}\label{SETUP:P.x.EST}
P (x)\leq \prod_{p\leq x} p =\textup{exp}((1+o(1))x),
\end{equation}and, by the Prime Number Theorem for arithmetic progressions, we obtain
\begin{align}\label{Density.asymp}
\sigma (x)&= \prod_{q\leq x} \bigg(1- \frac{2}{q}\bigg)=\prod_{\substack{a+|b|< p\leq x\\ p\equiv 3\,\text{(mod $4$)}}} \bigg(1- \frac{2}{p}\bigg)\notag\\
&= \prod_{\substack{p\leq a+|b|\\ p\equiv 3\,\text{(mod $4$)}}} \bigg(1- \frac{2}{p}\bigg)^{-1}
\prod_{\substack{p\leq x\\ p\equiv 3\,\text{(mod $4$)}}} \bigg(1- \frac{2}{p}\bigg)\sim \frac{C}{\log x},
\end{align} where $C=C(a,b)>0$ is a constant depending only on $a$ and $b$.

 For any integer $d$, we define
\begin{align*}
S_{x}(d)&:= \{n\in \mathbb{Z}:\ n-d \not \equiv \alpha\text{ (mod $p$) for any $\alpha\in I_{q}$ for any $q\leq x$}\},\\
S_{z,x}(d)&:= \{n\in \mathbb{Z}:\ n-d \not \equiv \alpha\text{ (mod $p$) for any $\alpha \in I_{q}$ for any $z<q\leq x$}\}.
\end{align*} We denote $S_x:= S_{x}(0)$. It is clear that $S_{x}(d)= S_{x}+ d$. Here $S_{x}+d:= \{s+d:\ s\in S_{x}\}$.

Let $\xi > 1$. We set
\[
\mathfrak{H}:=\bigg\{H\in\{1, \xi, \xi^{2},\ldots\}:\ \frac{2y}{x}\leq H\leq \frac{y}{\xi z}\bigg\}
\]and
\[
\mathfrak{H}'=\{H\in \mathfrak{H}: H=\xi^{j},\ j\text{ is even}\},\qquad \mathfrak{H}''=\{H\in \mathfrak{H}: H=\xi^{j},\ j\text{ is odd}\}.
\]It is clear that
\[
\bigsqcup_{H\in\mathfrak{H}}\Big(\frac{y}{\xi H}, \frac{y}{H}\Big]\subset \Big(z, \frac{x}{2}\Big],
\]and
\begin{equation}\label{H_range}
(\log x)^{\delta}\leq H\leq \frac{(\log x)^{1/2}}{\log\log x}\quad \qquad (H\in \mathfrak{H}).
\end{equation} For each $H\in \mathfrak{H}$, let
 \[
 \mathcal{Q}_{H}=\Big\{q: \frac{y}{\xi H}< q \leq \frac{y}{H}\Big\}.
 \]We have
\begin{equation}\label{QH.NU.ASYMPT}
\# \mathcal{Q}_{H}\sim \Big(1-\frac{1}{\xi}\Big)\frac{y}{2H\log x}.
\end{equation}Also, let
\[
\mathcal{Q}' = \bigcup_{H\in \mathfrak{H}'}\mathcal{Q}_{H},\qquad \mathcal{Q}'' = \bigcup_{H\in \mathfrak{H}''}\mathcal{Q}_{H},\qquad\text{and}\qquad \mathcal{Q}=\mathcal{Q}'\cup \mathcal{Q}''.
\]

We note that $\mathcal{Q}'\cap \mathcal{Q}''=\emptyset$, and if $q\in \mathcal{Q}$, then $q\in (z, x/2]$. It is clear that
\begin{equation}\label{SETUP:Q.EST}
\#\mathcal{Q}\leq \#\Big\{p: p \leq \frac{x}{2}\Big\}\leq \frac{x}{\log x}<y.
\end{equation} For $q\in \mathcal{Q}$, let $H_{q}$ be the unique element of $\mathfrak{H}$ such that
\[
\frac{y}{\xi H_{q}}< q\leq \frac{y}{H_{q}}.
\]

We denote by $\mathbf{d}$ a random residue class from $\mathbb{Z}/P\mathbb{Z}$, chosen with uniform probability, where we adopt the abbreviations
\[
P=P (z),\qquad \sigma=\sigma (z),\qquad \mathbf{S}'= S_{z}(\mathbf{d}),\qquad \mathbf{S}''= S_{z}(-N -\mathbf{d}).
\]Let
\[
6<M\leq 7.
\]For $H\in\mathfrak{H}$, we use the notation
\[
P_{1}=P (H^{M}),\qquad \sigma_{1}=\sigma (H^{M}),\qquad P_{2}=P (H^{M}, z),\qquad \sigma_{2}=\sigma (H^{M}, z),
\]and define
\begin{gather}
\mathbf{d}_{1}\equiv \mathbf{d}\ \text{(mod $P_{1}$)},\qquad
\mathbf{S}'_{1}= S_{H^{M}}(\mathbf{d}_{1}),\qquad \mathbf{S}''_{1}= S_{H^{M}}(-N-\mathbf{d}_{1}),\label{DECOMP.1}\\
\mathbf{d}_{2}\equiv \mathbf{d}\ \text{(mod $P_{2}$)},\qquad \mathbf{S}'_{2}= S_{H^{M}, z}(\mathbf{d}_{2}),
\qquad \mathbf{S}''_{2}= S_{H^{M}, z}(-N -\mathbf{d}_{2}),\label{DECOMP.2}
\end{gather}
with the convention that $\mathbf{d}_{1}\in \mathbb{Z}/P_{1}\mathbb{Z}$ and $\mathbf{d}_{2}\in \mathbb{Z}/P_{2}\mathbb{Z}$. Thus, $\mathbf{d_1}$ and $\mathbf{d_{2}}$ are each uniformly distributed, are independent of each other, and likewise $\mathbf{S}'_{1}$ and $\mathbf{S}'_{2}$ are independent, and $\mathbf{S}''_{1}$ and $\mathbf{S}''_{2}$ are independent. We also have the obvious relations
\begin{equation}\label{DECOMP.FINAL}
P=P_1 P_2,\qquad \sigma=\sigma_1\sigma_2,\qquad \mathbf{S}'=\mathbf{S}'_{1}\cap \mathbf{S}'_{2},\qquad \mathbf{S}''=\mathbf{S}''_{1}\cap \mathbf{S}''_{2}.
\end{equation} From \eqref{Density.asymp} we have
\begin{equation}\label{SIGMA.2}
\sigma_{2}^{-1}=\prod_{H^{M}< q \leq z}\bigg(1-\frac{2}{q}\bigg)^{-1}\sim \frac{\log z}{M\log H}\sim \frac{\log x}{M\log H}.
\end{equation} Therefore, for sufficiently large $x$, we have
\[
\sigma_{2}^{-1}\leq \frac{2\log x}{M \log H}\leq \frac{2 \log x}{M\delta \log\log x}.
\]

 For $q\in \mathcal{Q}$, let $\alpha_{q,1} = 0$ and $\alpha_{q,2} = c_{q}$. Then
\[
I_{q}=\{\alpha_{q, i}:  i=1,2 \}
\]and $\alpha_{q,i}\in [0,q-1]$. For $n\in \mathbb{Z}$, we set
\begin{align*}
\mathbf{AP}'(J; q, n)&:=\bigg(\bigsqcup_{i=1}^{2}\{n+\alpha_{q,i}+qh:\ 1\leq h \leq J\}\bigg)\cap \mathbf{S}'_{1},\\
\mathbf{AP}''(J; q, n)&:=\bigg(\bigsqcup_{i=1}^{2}\{n+\alpha_{q,i}-qh:\ 1\leq h \leq J\}\bigg)\cap \mathbf{S}''_{1}.
\end{align*} Let
\begin{equation}\label{Def_lambda}
\bm{\lambda}'(H; q, n)= \frac{1_{\mathbf{AP}'(KH; q, n)\subset \mathbf{S}'_{2}}}{\sigma_{2}^{\# \mathbf{AP}'(KH; q, n)}}\qquad\text{and}\qquad    \bm{\lambda}''(H; q, n)= \frac{1_{\mathbf{AP}''(KH; q, n)\subset \mathbf{S}''_{2}}}{\sigma_{2}^{\# \mathbf{AP}''(KH; q, n)}}.
\end{equation}

Suppose that $n\in [1,y]$, $q\in \mathcal{Q}'$, $i\in \{1,2\}$, and $1 \leq h \leq K H_{q}$. Then $n- \alpha_{q, i} - qh \leq y$ and
\[
n- \alpha_{q, i} - qh \geq 1 - (q-1) - q K H_{q}\geq 2- q - Ky > - (K+1)y.
\]Hence,
\begin{equation}\label{n.qh.range.1}
- (K+1)y < n- \alpha_{q, i} - qh\leq y.
\end{equation}Similarly, let $n\in [-y,-1]$, $q\in \mathcal{Q}''$, $i \in\{1,2\}$, and $1 \leq h \leq K H_{q}$. Then
\[
n - \alpha_{q,i}+ qh \geq -y- (q-1) +q > - y,
\]and
\[
n - \alpha_{q,i}+ qh \leq - 1 + q K H_{q} \leq - 1 + Ky < (K+1)y.
\]Thus, we obtain
\begin{equation}\label{n.qh.range.2}
-y \leq n - \alpha_{q,i}+ qh < (K+1)y.
\end{equation}

\section{Main Propositions}

\begin{proposition}\label{P1}
Let $10^{-6}< \delta < C(1/2)$. Then there exist constants $6<M<7$, $\xi>1$, $K>0$, $0 < \varepsilon < (M-6)/7$ such that for sufficiently large $x$ (with respect to $\delta$, $M$, $\xi$, $K$, and $\varepsilon$) there exist an integer $d$ \textup{mod $P(z)$} (a choice of $\mathbf{d}$) and non-empty sets $\mathcal{T}'\subseteq \mathcal{Q}'$ and $\mathcal{T}''\subseteq \mathcal{Q}''$ such that the following statements hold.

\textup{(i)} One has
\begin{equation}\label{S.EST}
\#(S'\cap [1, y])\leq 2\sigma y\qquad\text{and}\qquad \#(S''\cap [-y, -1])\leq 2\sigma y.
\end{equation}

\textup{(ii)} For all $q\in \mathcal{T}'$ one has
\begin{equation}\label{L1:lambda.1}
\sum_{-(K+1)y< n \leq y} \lambda ' (H_{q}; q, n) = \left(1+O\left(\frac{1}{(\log x)^{\delta (1+\varepsilon)}}\right)\right)(K+2)y.
\end{equation}

For all $q\in \mathcal{T}''$ one has
\begin{equation}\label{L1:lambda.2.N}
\sum_{-y\leq n < (K+1) y} \lambda ''(H_{q}; q, n) = \left(1+O\left(\frac{1}{(\log x)^{\delta (1+\varepsilon)}}\right)\right)(K+2)y.
\end{equation}

\textup{(iii)} For each $i\in\{1,2\}$, there exists a set $V'_{i}\subset (S '\cap [1,y])$ such that
\begin{equation}\label{V.1.DOPOLNENIYE}
\# ((S'\cap [1, y]) \setminus V'_{i}) \leq \frac{x}{100 \log x}
\end{equation}and for any $n\in V'_{ i}$, one has
\begin{equation}\label{L1:lambda.1.struya}
\sum_{q\in \mathcal{T}'} \sum_{h\leq K H_{q}}\lambda ' (H_{q}; q, n-\alpha_{q,i}-qh)=\bigg(C' + O\bigg(\frac{1}{(\log x)^{\delta (1+\varepsilon)}}\bigg)\bigg) (K+2)y
\end{equation}for some quantity $C'$ independent of $n$ and $i$ with
\begin{equation}\label{C.1.DIAPAZON}
 \frac{10 ^{2\delta}}{2} \leq C' \leq 50 .
\end{equation}
\textup{(iv)} For each $i\in\{1,2\}$, there exists a set $V''_{i}\subset (S ''\cap [-y,-1])$ such that
\begin{equation}\label{V.2.DOPOLNENIYE}
\# ((S''\cap [-y, -1]) \setminus V''_{ i}) \leq \frac{x}{100 \log x}
\end{equation}and for any $n\in V''_{ i}$, one has
\begin{equation}\label{L1:lambda.2.struya}
\sum_{q\in \mathcal{T}''} \sum_{h\leq K H_{q}}\lambda '' (H_{q}; q, n-\alpha_{q,i}+qh)=\bigg(C'' + O\bigg(\frac{1}{(\log x)^{\delta (1+\varepsilon)}}\bigg)\bigg) (K+2)y
\end{equation}for some quantity $C''$ independent of $n$ and $i$ with
\[
 \frac{10 ^{2\delta}}{2} \leq C'' \leq 50 .
\]
\end{proposition}

\begin{proposition}\label{P2}
The following statements hold.

\textup{(i)} One has
\begin{align}
\mathbb{E} \#(\mathbf{S}'\cap [1,y])&=\sigma y,\qquad\quad\quad\ \
\mathbb{E} \big(\#(\mathbf{S}'\cap [1,y])\big)^{2}= \bigg(1+ O\bigg(\frac{1}{\log y}\bigg)\bigg)(\sigma y)^{2},\label{S.1.OSNOVA}\\
\mathbb{E} \#(\mathbf{S}''\cap [-y,-1])&=\sigma y,\qquad\quad
\mathbb{E} \big(\#(\mathbf{S}''\cap [-y,-1])\big)^{2}= \bigg(1+ O\bigg(\frac{1}{\log y}\bigg)\bigg)(\sigma y)^{2}.\label{S.2.OSNOVA}
\end{align}

\textup{(ii)} For every $H\in \mathfrak{H}'$ and $j\in \{1, 2\}$, we have
\begin{equation}\label{L1:lambda}
\mathbb{E}\sum_{q\in \mathcal{Q}_{H}} \Bigg(\sum_{-(K+1)y< n\leq y}\bm{\lambda}'(H; q, n)\Bigg)^{j}= \bigg(1 +
 O\bigg(\frac{\log H}{H^{M-2}}\bigg)\bigg) \big((K+2)y\big)^{j}\# \mathcal{Q}_{H}.
\end{equation}For every $H\in \mathfrak{H}''$ and $j\in \{1, 2\}$, we have
\begin{equation}\label{L1:lambda.2}
\mathbb{E}\sum_{q\in \mathcal{Q}_{H}} \Bigg(\sum_{-y\leq n< (K+1)y}\bm{\lambda}''(H; q, n)\Bigg)^{j}= \bigg(1 +
 O\bigg(\frac{\log H}{H^{M-2}}\bigg)\bigg) \big((K+2)y\big)^{j}\# \mathcal{Q}_{H}.
\end{equation}

\textup{(iii)} For every $H\in \mathfrak{H}'$, $i\in \{1, 2\}$, and $j\in \{1,2\}$, we have
\begin{align}\label{L3:lamda.AP.1}
\mathbb{E}\sum_{n\in \mathbf{S}'\cap [1, y]} \bigg(\sum_{q\in \mathcal{Q}_{H}} \sum_{h\leq KH}\bm{\lambda}'(H; q, &n-\alpha_{q, i}-qh)\bigg)^{j}\notag\\
 &= \bigg(1 + O\bigg(\frac{\log H}{H^{M-2}}\bigg)\bigg) \bigg(\frac{\#\mathcal{Q}_{H}[KH]}{\sigma_{2}}\bigg)^{j}\sigma y.
\end{align}

For every $H\in \mathfrak{H}''$, $i\in \{1, 2\}$, and $j\in \{1,2\}$, we have
\begin{align}\label{L3:lamda.AP.2}
\mathbb{E}\sum_{n\in \mathbf{S}''\cap [-y, -1]} \bigg(\sum_{q\in \mathcal{Q}_{H}} \sum_{h\leq KH}\bm{\lambda}''&(H; q, n-\alpha_{q, i}+qh)\bigg)^{j}\notag\\
 &= \bigg(1 + O\bigg(\frac{\log H}{H^{M-2}}\bigg)\bigg) \bigg(\frac{\#\mathcal{Q}_{H}[KH]}{\sigma_{2}}\bigg)^{j}\sigma y.
\end{align}

\end{proposition}

The scheme of proof of Theorem \ref{T1} will be as follows. We will deduce Theorem \ref{T1} from Proposition \ref{P1}. After that we will deduce Proposition \ref{P1} from Proposition \ref{P2}. The proof of Proposition \ref{P2} will be given in the last sections.

\section{Deduction of Theorem \ref{T1} from Proposition \ref{P1}}\label{SECTION.DEDUCE}

We will use the following hypergraph covering lemma.
\begin{lemma}\label{L2}
Suppose that $0<\delta \leq 1/2$ and $K_{0}\geq 1$, and let $y\geq y_{0}(\delta, K_{0})$ with $y_{0}(\delta, K_{0})$ sufficiently large, and let $V$ be a finite set with $\#V\leq y$. Let $1\leq s \leq y$, and suppose that $\mathbf{e}_{1},\ldots, \mathbf{e}_{s}$ are random subsets of $V$ satisfying the following:
\begin{align}
\#\mathbf{e}_{i}&\leq \frac{K_{0}(\log y)^{1/2}}{\log\log y}\qquad (1\leq i\leq s),\label{L3:I}\\
\mathbb{P}(v\in \mathbf{e}_{i})&\leq y^{-1/2 - 1/100}\qquad (v\in V,\ 1\leq i \leq s),\label{L3:II}\\
\sum_{i=1}^{s}\mathbb{P}(v, v' \in \mathbf{e}_{i})&\leq y^{-1/2}\qquad (v, v' \in V,\ v\neq v'),\label{L3:III}\\
\bigg|\sum_{i=1}^{s} \mathbb{P}(v\in \mathbf{e}_{i}) - C_1 \bigg|&\leq \eta\qquad (v\in V),\label{L3:IV}
\end{align}where $C_1$ and $\eta$ satisfy
\[
10^{2\delta}\leq C_1\leq 100,\qquad \eta \geq \frac{1}{(\log y)^{\delta}\log\log y}.
\]Then there are subsets $e_i$ of $V$, $1\leq i \leq s$, with $e_i$ being in the support of $\mathbf{e}_{i}$ for every $i$, and such that
\[
\# \bigg(V\setminus \bigcup_{i=1}^{s} e_i\bigg)\leq C_{0}\eta \#V,
\]where $C_0 >0$ is an absolute constant.
\end{lemma}
\begin{proof}
This is \cite[Lemma 3.1]{FKMPT}.
\end{proof}

We put
\[
V' = \bigcap_{i=1}^{2} V'_{ i}.
\]From part \textup{(iii)} of Proposition \ref{P1}, we obtain $V' \subset (S' \cap [1, y])$,
\[
\#((S' \cap [1, y])\setminus V') \leq \sum_{i=1}^{2}\#((S' \cap [1, y])\setminus V'_{ i})\leq \frac{ x}{50 \log x},
\]and if $n\in V'$, then for each $i\in\{1,2\}$ the statement \eqref{L1:lambda.1.struya} holds.

For each $q\in \mathcal{T}'$, we choose a random integer $\mathbf{n}'_{q}$ with probability density function
\begin{equation}\label{Def_n}
\mathbb{P}(\mathbf{n}'_{q}=n)=\frac{\lambda '(H_{q}; q, n)}{\sum_{-(K+1)y<t\leq y}\lambda '(H_{q}; q, t)}\qquad
(-(K+1)y< n \leq y).
\end{equation} Note that by \eqref{L1:lambda.1} the denominator is non-zero, so that this is a well-defined probability distribution. For $q\in \mathcal{T}'$, we define the random sets
\[
\mathbf{e}'_{q, i} = V' \cap \{\mathbf{n}'_{q}+ \alpha_{q,i} + qh: 1\leq h\leq KH_{q}\},\quad i=1, 2,
\]and put
\[
\mathbf{e}'_{q} = \bigsqcup_{i= 1}^{2} \mathbf{e}'_{q, i}.
\] In particular, we obtain that $\mathbf{e}'_{q} \subset V'$ for any $q\in \mathcal{T}'$.

We are going to apply Lemma \ref{L2} with $V=V'$, $\{\mathbf{e}_{1},\ldots, \mathbf{e}_{s}\}=\{\mathbf{e}'_{q}:\ q\in \mathcal{T}'\}$, $s=\# \mathcal{T}'$, $K_{0}=2K$, and
\begin{equation}\label{ETA.DEF}
\eta=\frac{1}{(\log y)^{\delta}\log \log y}.
\end{equation} We note that $s\geq 1$, since $\mathcal{T}'\neq \emptyset$, and by \eqref{SETUP:Q.EST} we have $s\leq \#\mathcal{Q}\leq y.$

If $q\in \mathcal{T}'$, then from \eqref{H_range} we have
\[
\#\mathbf{e}'_{q}\leq 2KH_{q}\leq \frac{2K(\log x)^{1/2}}{\log \log x}\leq   \frac{K_{0}(\log y)^{1/2}}{\log \log y}.
\] Hence, \eqref{L3:I} holds.

For $n\in V'$ and $q\in \mathcal{T}'$, we have from \eqref{n.qh.range.1}, \eqref{Def_n}, \eqref{L1:lambda.1}, \eqref{Def_lambda}, and \eqref{SIGMA.2} that
\begin{align}
\mathbb{P}(n\in \mathbf{e}'_{q})&=\sum_{i= 1}^{2}\sum_{1\leq h \leq KH_{q}} \mathbb{P}(\mathbf{n}'_{q}=n -\alpha_{q,i}-qh)\notag\\
&=\sum_{i=1}^{2}\sum_{1\leq h \leq K H_{q}} \frac{\lambda '(H_{q}; q, n-\alpha_{q,i}-qh)}{\sum_{-(K+1)y <t\leq y}\lambda '(H_{q}; q, t)}\notag\\
&\ll \frac{1}{y} H_{q} \sigma_{2}^{-2K H_q}\ll y^{-9/10},\label{PROB.n.Eq.1}
\end{align} which gives \eqref{L3:II}.

For $n\in V'$, we have from \eqref{Def_n}, \eqref{L1:lambda.1.struya}, \eqref{L1:lambda.1}, and \eqref{C.1.DIAPAZON} that
\begin{align*}
\sum_{q\in \mathcal{T}'} \mathbb{P}(n\in \mathbf{e}'_{q})&=
\sum_{q\in \mathcal{T}'}\sum_{1 \leq i \leq 2} \sum_{h \leq KH_{q}} \mathbb{P}(\mathbf{n}'_{q}=n-\alpha_{q,i}-qh)\\
&=\sum_{1 \leq i \leq 2}\sum_{q\in \mathcal{T}'} \sum_{h \leq KH_{q}}
\frac{\lambda '(H_{q}; q, n-\alpha_{q,i}-qh)}{\sum_{-(K+1)y<t\leq y}\lambda '(H_{q}; q, t)}\\
&= 2C' + O\big((\log x)^{-\delta(1+ \varepsilon)}\big)
\end{align*}and $10^{2\delta}\leq 2C'\leq 100$. Thus, \eqref{L3:IV} follows.

We now turn to \eqref{L3:III}. For any distinct $n, m \in V'$,  we have
\begin{align*}
\sum_{q\in \mathcal{T}'} \mathbb{P}(n, m \in \mathbf{e}'_{q})\leq \sum_{q\in \mathcal{T}'}\sum_{1\leq i \leq 2} \mathbb{P}(n, m \in \mathbf{e}'_{q,i})+
\sum_{q\in \mathcal{T}'}\sum_{\substack{1 \leq i, j\leq 2\\ i\neq j}}\mathbb{P}(n\in \mathbf{e}'_{q,i}, &m\in \mathbf{e}'_{q,j})\\
&=A_1 + A_2.
\end{align*}We first estimate the sum $A_{1}$. Put $N_{1}=n-m$. If both $n$, $m$ belong to some $\mathbf{e}_{q,i}$, then $q$ divides $N_{1}$. But $0<|N_{1}|<y$ and $q>z> y^{3/4}$, hence there is at most one such $q$. By \eqref{PROB.n.Eq.1}, we obtain
\begin{equation}\label{L:EST.A.1}
A_{1}\ll y^{-0.9}.
\end{equation}

Now we estimate the sum $A_{2}$. If $n\in \mathbf{e}_{q,i}$ and $m\in \mathbf{e}_{q, j}$ for some $q$ and $i\neq j$, then $N_{1} \equiv c_{q}$ (mod $q$) or $N_{1} \equiv - c_{q}$ (mod $q$), and hence $q$ divides $aN_{1}+b$ or $a(-N_{1})+ b$. But these both numbers are non-zero since $a \nmid b$, and therefore the number of such $q$ is $\ll \log y$. By \eqref{PROB.n.Eq.1}, we obtain
\begin{equation}\label{L:EST.A.2}
A_{2} \ll y^{-0.9} \log y.
\end{equation} From \eqref{L:EST.A.1} and \eqref{L:EST.A.2} we get
\[
\sum_{q\in \mathcal{T}'} \mathbb{P}(n, m \in \mathbf{e}'_{q}) \ll y^{-0.9} \log y < y^{-0.5},
\] which gives \eqref{L3:III}.

Thus, all hypotheses of Lemma \ref{L2} hold, and for each $q\in \mathcal{T}'$ there is a number $n'_{q}$ such that if we put
\[
e'_{q, i} = V' \cap \{n'_{q}+ \alpha_{q,i} + qh: 1\leq h\leq KH_{q}\},\quad i=1,2,
\]and
\[
e'_{q} = \bigsqcup_{i=1}^{2} e'_{q, i},
\]then we have
\[
\# \Big(V'\setminus \bigcup_{q\in \mathcal{T}'} e'_{q}\Big)\leq C_{0}\eta\#V'.
\]By \eqref{S.EST}, we have
\[
\#V'\leq 2 \sigma y \ll \frac{x (\log x)^{\delta}}{\log x}.
\] Thus, by \eqref{ETA.DEF},
\[
\# \Big(V'\setminus \bigcup_{q\in \mathcal{T}'} e'_{q}\Big) \ll \frac{x}{\log x \log\log x}
< \frac{ x}{50 \log x}.
\]

Similarly, we put
\[
V'' = \bigcap_{i=1}^{2} V''_{ i}.
\]From part \textup{(iv)} of Proposition \ref{P1}, we obtain $V'' \subset (S'' \cap [-y, -1])$,
\[
\#((S'' \cap [-y, -1])\setminus V'') \leq  \frac{ x}{50\log x},
\]and if $n\in V''$, then for each $i\in\{1,2\}$ the statement \eqref{L1:lambda.2.struya} holds.

For each $q\in \mathcal{T}''$, we choose a random integer $\mathbf{n}''_{q}$ with probability density function
\[
\mathbb{P}(\mathbf{n}''_{q}=n)=\frac{\lambda '(H_{q}; q, n)}{\sum_{-y\leq t< (K+1)y}\lambda ''(H_{q}; q, t)}\qquad
(-y\leq n < (K+1)y).
\] By \eqref{L1:lambda.2.N} the denominator is non-zero, so that this is a well-defined probability distribution. For $q\in \mathcal{T}''$, we define the random sets
\[
\mathbf{e}''_{q, i} = V'' \cap \{\mathbf{n}''_{q}+ \alpha_{q,i} - qh: 1\leq h\leq KH_{q}\},\quad i=1, 2,
\]and
\[
\mathbf{e}''_{q} = \bigsqcup_{i=1}^{2} \mathbf{e}''_{q, i},
\]In particular, we have $\mathbf{e}''_{q} \subset V''$ for any $q\in \mathcal{T}''$.

We are going to apply Lemma \ref{L2} with $V=V''$, $\{\mathbf{e}_{1},\ldots, \mathbf{e}_{s}\}=\{\mathbf{e}''_{q}:\ q\in \mathcal{T}''\}$, $s=\# \mathcal{T}''$, $K_{0}=2K$, and $\eta$ given by \eqref{ETA.DEF}.
Arguing as above, we see that all hypotheses of Lemma \ref{L2} hold, and hence for each $q\in \mathcal{T}''$ there is a number $n''_{q}$ such that if we set
\[
e''_{q, i} = V'' \cap \{n''_{q}+ \alpha_{q,i} - qh: 1\leq h\leq KH_{q}\},\quad i=1,2,
\]and
\[
e''_{q} = \bigsqcup_{i=1}^{2} e''_{q, i},
\]then we have
\[
\# \Big(V''\setminus \bigcup_{q\in \mathcal{T}''} e''_{q}\Big)\leq C_{0}\eta\#V''< \frac{ x}{50 \log x}.
\]

We recall that $\mathcal{T}' \cap \mathcal{T}'' = \emptyset$ and if $q\in \mathcal{Q}$, then $z< q \leq x/2$. By the Chinese Remainder Theorem, we take $d\equiv n'_{q}$ (mod $q$) for all $q\in \mathcal{T}'$ and $d \equiv -N - n''_{q}$ (mod $q$) for all $q\in \mathcal{T}''$. Therefore
\begin{align*}
\# (S_{x/2}(d)\cap [1,y])&\leq \# ((S'\cap [1, y])\setminus V')+
\# \Big(V'\setminus \bigcup_{q\in \mathcal{T}'}e'_{q}\Big)\\
&\leq \frac{ x}{50 \log x} + \frac{x}{50 \log x} = \frac{ x}{25 \log x},
\end{align*}and
\begin{align*}
\# (S_{x/2}(-N-d)\cap [-y,-1])&\leq \# ((S''\cap [-y, -1])\setminus V'')+
\# \Big(V''\setminus \bigcup_{q\in \mathcal{T}''}e''_{q}\Big)\\
&\leq \frac{x}{50 \log x} + \frac{x}{50 \log x} = \frac{x}{25 \log x}.
\end{align*}

Let us denote
\begin{align*}
\mathcal{A}'&:= S_{x/2}(d)\cap [1,y],\qquad &&\mathcal{A}'':= S_{x/2}(-N-d)\cap [-y,-1],\\
\mathcal{D}' &:=\{q: x/2< q\leq (3x)/4\},\qquad &&\mathcal{D}'' :=\{q: (3x)/4< q\leq x\}.
\end{align*}
 Then we have
\[
\# \mathcal{D}' > \frac{x}{10\log x}> \#\mathcal{A}',\qquad \# \mathcal{D}'' > \frac{x}{10\log x}> \#\mathcal{A}''.
\] Hence, we may pair up each element $a\in \mathcal{A}'$ with a unique prime $q(a)\in \mathcal{D}'$ and pair up each element $a\in \mathcal{A}''$ with a unique prime $q(a)\in \mathcal{D}''$. We take $d\equiv a$ (mod $q(a)$) for every $a\in \mathcal{A}'$ and $d\equiv -N - a$ (mod $q(a)$) for every $a\in \mathcal{A}''$ (applying again the Chinese Remainder Theorem). We obtain
\[
S_{x}(d)\cap [1,y]=\emptyset \qquad\text{and}\qquad S_{x}(-N-d)\cap [-y,-1]=\emptyset.
\]

Let $x=(\log N)/4$ and let $N$ be sufficiently large. By \eqref{SETUP:P.x.EST} and \eqref{SETUP:Def.y},  we have
\[
 P(x)\leq N^{1/3}\qquad \text{and}\qquad y \sim \frac{1}{4} \log N (\log\log N)^{\delta}.
\] Let $d_{1}$ be such that $d_{1}\equiv d$ (mod $P (x)$) and $d_{1}\in [-0.3 N, -0.2 N]$. We have $S_{x}(d_{1})\cap [1,y]=\emptyset$, $S_{x}(-N -d_{1})\cap [-y,-1]=\emptyset$,  and hence
\[
S_{x}\cap (d_2 + [1,y]) = \emptyset,\qquad\text{and}\qquad S_{x}\cap (N-d_2 + [-y,-1]) = \emptyset ,
\]where $d_{2} = -d_{1}$. We set $I_{1}=\{d_{2}+1,\ldots, d_{2}+y\}$ and $I_{2}=\{N-d_{2}- y,\ldots, N-d_{2}-1\}$. Then $I_{1}$ and $I_{2}$ are strings of consecutive positive integers $n \in [0.2 N, 0.8 N]$ of length $y \gg \log N (\log\log N)^{\delta}$ and for any $n\in I_{1}\cup I_{2}$ there is a prime $q\leq (\log N)/4$ such that $q$ divides $n(an+b)$. Since $q\equiv 3$ (mod $4$) at least one of $n$ or $an+b$ does not lie in $\mathcal{R}$. Finally, we take $n_{1} = d_{2} + [y/2]$ and $n_{2} = N - d_{2} - [y/2]$. This completes the proof of Theorem \ref{T1} assuming Proposition \ref{P1}.

\section{Deduction of Proposition \ref{P1} from Proposition \ref{P2}}

 From part \textup{(i)} of Proposition \ref{P2} we have
\[
\mathbb{E}(\#(\mathbf{S}'\cap [1, y]) - \sigma y)^{2}\ll \frac{(\sigma y)^{2}}{\log y}\qquad\text{and}\qquad
\mathbb{E}(\#(\mathbf{S}''\cap [-y, -1]) - \sigma y)^{2}\ll \frac{(\sigma y)^{2}}{\log y}.
\]Hence by Chebyshev's inequality, we see that
\begin{align}
\mathbb{P}( \#(\mathbf{S}'\cap [1, y])\leq 2\sigma y)&= 1 - O\left(\frac{1}{\log x}\right),\label{PROB.S.EST}\\
\mathbb{P}( \#(\mathbf{S}''\cap [-y, -1])\leq 2\sigma y)&= 1 - O\left(\frac{1}{\log x}\right).\label{PROB.S.EST.2}
\end{align}

Fix $H\in \mathfrak{H}'$. From \eqref{L1:lambda} we have
\begin{equation}\label{SUMMA.LAMBDA.1.M}
\mathbb{E}\sum_{q\in \mathcal{Q}_{H}}\Bigg(\sum_{-(K+1)y< n\leq y}\bm{\lambda}'(H;q,n)- (K+2)y\Bigg)^{2}\ll \frac{y^{2}\#\mathcal{Q}_{H}\log H}{H^{M-2}}\ll \frac{y^{2}\#\mathcal{Q}_{H}}{H^{M-2-\varepsilon}}.
\end{equation}We put
\[
\bm{\mathcal{T}}'_{H}=\Big\{q\in \mathcal{Q}_{H}: \Big|\sum_{-(K+1)y<n\leq y}\bm{\lambda}'(H;q,n)- (K+2)y\Big|\leq \frac{y}{H^{1+\varepsilon}}\Big\}.
\] We have
\begin{align}
\mathbb{E}&\# (\mathcal{Q}_{H}\setminus \bm{\mathcal{T}}'_{H})=\mathbb{E}\sum_{q\in \mathcal{Q}_{H}\setminus \bm{\mathcal{T}}'_{H}} 1\notag\\
&=\mathbb{E}\sum_{q\in \mathcal{Q}_{H}\setminus \bm{\mathcal{T}}'_{H}} \frac{\bigg(\sum_{-(K+1)y< n\leq y}\bm{\lambda}'(H;q,n)- (K+2)y\bigg)^{2}}{\bigg(\sum_{-(K+1)y< n\leq y}\bm{\lambda}'(H;q,n)- (K+2)y\bigg)^{2}}\notag\\
&\leq \frac{H^{2+2\varepsilon}}{y^{2}}\mathbb{E}\sum_{q\in \mathcal{Q}_{H}}\Bigg(\sum_{-(K+1)y< n\leq y}\bm{\lambda}'(H;q,n)- (K+2)y\Bigg)^{2}
\ll \frac{\# \mathcal{Q}_{H}}{H^{M-4-3\varepsilon}}.\label{EXP.R.H.NU.1}
\end{align}By Markov's inequality, we have
\[
\mathbb{P}\bigg(\# (\mathcal{Q}_{H}\setminus \bm{\mathcal{T}}'_{H})\leq \frac{\# \mathcal{Q}_{H}}{H^{M-4-4\varepsilon}}\bigg)=1- O(H^{-\varepsilon}).
\]We observe that for any $\beta>0$ we have (see \ref{H_range})
\begin{equation}\label{L2:H_alpha}
\sum_{H\in \mathfrak{H}} H^{-\beta}\leq \frac{1}{\big((\log x)^{\delta}\big)^{\beta}}+
\frac{1}{\big(\xi(\log x)^{\delta}\big)^{\beta}}+\frac{1}{\big(\xi^{2}(\log x)^{\delta}\big)^{\beta}}+\ldots\ll_{\beta} (\log x)^{-\delta \beta}.
\end{equation}Hence, with probability $1-O((\log x)^{-\delta\varepsilon})$ the relation
\begin{equation}\label{BASIC.I}
\# (\mathcal{Q}_{H}\setminus \bm{\mathcal{T}}'_{H})\leq \frac{\# \mathcal{Q}_{H}}{H^{M-4-4\varepsilon}}
\end{equation}holds for every $H\in\mathfrak{H}'$ simultaneously. We put
\[
\bm{\mathcal{T}}'=\bigcup_{H\in \mathfrak{H}'} \bm{\mathcal{T}}'_{H}.
\]Since $H\geq (\log x)^{\delta}$ for any $H\in \mathfrak{H}$, we have for any $q\in \bm{\mathcal{T}}'$
\begin{equation}\label{SUM.LAMBDA.1}
\sum_{-(K+1)y< n \leq y} \bm{\lambda}' (H_{q}; q, n) = \bigg(1+O\bigg(\frac{1}{(\log x)^{\delta (1+\varepsilon)}}\bigg)\bigg)(K+2)y.
\end{equation}

Fix $H\in \mathfrak{H}''$. From \eqref{L1:lambda.2} we have
\begin{equation}\label{SUMMA.LAMBDA.2.M}
\mathbb{E}\sum_{q\in \mathcal{Q}_{H}}\Bigg(\sum_{-y\leq n< (K+1)y}\bm{\lambda}''(H;q,n)- (K+2)y\Bigg)^{2}\ll \frac{y^{2}\#\mathcal{Q}_{H}\log H}{H^{M-2}}\ll \frac{y^{2}\#\mathcal{Q}_{H}}{H^{M-2-\varepsilon}}.
\end{equation}We put
\[
\bm{\mathcal{T}}''_{H}=\Big\{q\in \mathcal{Q}_{H}: \Big|\sum_{-y\leq n < (K+1)y}\bm{\lambda}''(H;q,n)- (K+2)y\Big|\leq \frac{y}{H^{1+\varepsilon}}\Big\}.
\] Similarly, we have
\begin{equation}\label{STRIKE.2}
\mathbb{E}\# (\mathcal{Q}_{H}\setminus \bm{\mathcal{T}}''_{H})
\ll \frac{\# \mathcal{Q}_{H}}{H^{M-4-3\varepsilon}},
\end{equation}and with probability $1-O((\log x)^{-\delta\varepsilon})$ the relation
\begin{equation}\label{BASIC.II}
\# (\mathcal{Q}_{H}\setminus \bm{\mathcal{T}}''_{H})\leq \frac{\# \mathcal{Q}_{H}}{H^{M-4-4\varepsilon}}
\end{equation}holds for every $H\in\mathfrak{H}''$ simultaneously. We put
\[
\bm{\mathcal{T}}''=\bigcup_{H\in \mathfrak{H}''} \bm{\mathcal{T}}''_{H}.
\]For any $q\in \bm{\mathcal{T}}''$ we have
\begin{equation}\label{SUM.LAMBDA.2}
\sum_{-y\leq n < (K+1)y} \bm{\lambda}'' (H_{q}; q, n) = \bigg(1+O\bigg(\frac{1}{(\log x)^{\delta (1+\varepsilon)}}\bigg)\bigg)(K+2)y.
\end{equation}

We work on part \textup{(iii)} of Proposition \ref{P1} using part \textup{(iii)} of Proposition \ref{P2} in a similar fashion to previous arguments. Fix $H\in \mathfrak{H}'$ and $i\in\{1,2\}$. From \eqref{L3:lamda.AP.1} we have
\begin{align*}
\mathbb{E}&\sum_{n\in \mathbf{S}'\cap [1,y]} \Bigg( \sum_{q\in \mathcal{Q}_{H}}\sum_{h\leq KH}\bm{\lambda}'(H;q, n-\alpha_{q, i}-qh) - \frac{\#\mathcal{Q}_{H} [KH]}{\sigma_{2}}\Bigg)^{2}\\
&\ll \frac{\log H}{H^{M-2}} \bigg(\frac{\#\mathcal{Q}_{H}[KH]}{\sigma_2}\bigg)^{2}\sigma y
\ll \frac{1}{H^{M-2 - \varepsilon}} \bigg(\frac{\#\mathcal{Q}_{H}[KH]}{\sigma_2}\bigg)^{2}\sigma y.
\end{align*}We put
\begin{align}
\bm{\mathcal{E}}'_{H, i}=\Bigg\{&n\in \mathbf{S}'\cap [1,y]:\notag\\
 &\Bigg|\sum_{q\in \mathcal{Q}_{H}}\sum_{h\leq K H}\bm{\lambda}'(H;q, n-\alpha_{q, i}-qh) - \frac{\#\mathcal{Q}_{H} [KH]}{\sigma_{2}}\Bigg|\geq \frac{\#\mathcal{Q}_{H} [KH]}{\sigma_{2} H^{1+ \varepsilon}}\Bigg\}.\label{E.1.DEF}
\end{align}Since $M>6$ and $\varepsilon < (M-6)/7$, we have
\begin{align*}
&\mathbb{E} \#\bm{\mathcal{E}}'_{H, i}=\mathbb{E}\sum_{n\in \bm{\mathcal{E}}'_{H, i}} 1\\
&=
\mathbb{E}\sum_{n\in \bm{\mathcal{E}}'_{H,i}}\frac{\big(\sum_{q\in \mathcal{Q}_{H}}\sum_{h\leq KH}\bm{\lambda}'(H;q, n-\alpha_{q,i}-qh) - (\#\mathcal{Q}_{H} [KH])/\sigma_{2}\big)^{2}}{\big(\sum_{q\in \mathcal{Q}_{H}}\sum_{h\leq KH}\bm{\lambda}'(H;q, n-\alpha_{q,i}-qh) - (\#\mathcal{Q}_{H} [KH])/\sigma_{2}\big)^{2}}\\
&\leq \frac{\sigma_{2}^{2}H^{2 +2\varepsilon}}{(\#\mathcal{Q}_{H}[KH])^{2}}
\mathbb{E}\sum_{n\in \mathbf{S}'\cap [1,y]}\Bigg(\sum_{q\in \mathcal{Q}_{H}}\sum_{h\leq KH}\bm{\lambda}'(H;q, n-\alpha_{q,i}-qh) - \frac{\#\mathcal{Q}_{H} [KH]}{\sigma_{2}}\Bigg)^{2}\\
&\qquad\qquad\qquad\qquad\qquad\qquad\qquad\qquad\qquad\qquad\qquad\qquad\quad\quad\ \ \
\ll \frac{\sigma y}{H^{M-4-3\varepsilon}}\ll \frac{\sigma y}{H^{2}}.
\end{align*}By Markov's inequality, we have
\begin{equation}\label{MARKOV.1}
\mathbb{P}\Big(\# \bm{\mathcal{E}}'_{H,i}\leq \frac{\sigma y}{H^{1+\varepsilon}}\Big)= 1-O\Big(\frac{1}{H^{1- \varepsilon}}\Big).
\end{equation}

We next estimate the contribution from ``bad'' primes $q\in \mathcal{Q}_{H}\setminus \bm{\mathcal{T}}'_{H}$. For any $h\leq KH$, by the Cauchy-Schwarz inequality we have
\begin{align}
\mathbb{E}\sum_{n\in \mathbf{S}'\cap [1,y]} &\sum_{q\in \mathcal{Q}_{H}\setminus \bm{\mathcal{T}}'_{H}} \bm{\lambda}'(H; q, n-\alpha_{q,i}-qh)\leq \big(\mathbb{E}\# (\mathcal{Q}_{H}\setminus \bm{\mathcal{T}}'_{H})\big)^{1/2}\cdot\notag\\
&\cdot\Bigg(\mathbb{E} \sum_{q\in \mathcal{Q}_{H}\setminus \bm{\mathcal{T}}'_{H}} \bigg(\sum_{n=1}^{y} \bm{\lambda}'(H; q, n-\alpha_{q, i}-qh)\bigg)^{2}\Bigg)^{1/2}.\label{CAUSHY.LAMBDA.1}
\end{align}Given $q\in \mathcal{Q}_{H}\setminus \bm{\mathcal{T}}'_{H}$, we have
\[
\bigg|\sum_{n=1}^{y} \bm{\lambda}'(H; q, n-\alpha_{q,i}-qh)\bigg|\leq \bigg|\sum_{n=1}^{y} \bm{\lambda}'(H; q, n-\alpha_{q, i}-qh)- (K+2)y\bigg| + (K+2)y.
\] Applying \eqref{n.qh.range.1}, we obtain
\begin{align*}
\bigg|\sum_{n=1}^{y} \bm{\lambda}'(H; q, n&-\alpha_{q, i}-qh)- (K+2)y\bigg|\\
 &\leq\max\Bigg((K+2)y, \bigg|\sum_{-(K+1)y<n\leq y} \bm{\lambda}'(H; q, n)- (K+2)y\bigg|\Bigg)\\
 &\leq\bigg|\sum_{-(K+1)y<n\leq y} \bm{\lambda}'(H; q, n)- (K+2)y\bigg|+ (K+2)y.
\end{align*}

Since $(a+b)^{2}\leq 2 (a^{2}+b^{2})$, we have
\[
\bigg(\sum_{n=1}^{y} \bm{\lambda}'(H; q, n-\alpha_{q,i}-qh)\bigg)^{2}\leq 2\bigg(\sum_{-(K+1)y<n\leq y} \bm{\lambda}'(H; q, n)- (K+2)y\bigg)^{2}+ 8(K+2)^{2}y^{2}.
\]Applying \eqref{SUMMA.LAMBDA.1.M} and \eqref{EXP.R.H.NU.1}, we obtain
\begin{align*}
\mathbb{E} \sum_{q\in \mathcal{Q}_{H}\setminus \bm{\mathcal{T}}'_{H}} &\bigg(\sum_{n=1}^{y} \bm{\lambda}'(H; q, n-\alpha_{q,i}-qh)\bigg)^{2}\\
 &\leq 2\mathbb{E} \sum_{q\in \mathcal{Q}_{H}\setminus \bm{\mathcal{T}}'_{H}} \bigg(\sum_{-(K+1)y<n\leq y} \bm{\lambda}'(H; q, n)- (K+2)y\bigg)^{2}\\
 &+ 8\big((K+2)y\big)^{2}\mathbb{E}\# (\mathcal{Q}_{H}\setminus \bm{\mathcal{T}}'_{H})
 \ll \frac{y^{2}\#\mathcal{Q}_{H}}{H^{M-2 -\varepsilon}} + \frac{y^{2}\#\mathcal{Q}_{H}}{H^{M-4 - 3\varepsilon}}\ll
 \frac{y^{2}\#\mathcal{Q}_{H}}{H^{M-4-3\varepsilon}}.
\end{align*} We see from \eqref{CAUSHY.LAMBDA.1} and \eqref{EXP.R.H.NU.1} that
\[
\mathbb{E}\sum_{n\in \mathbf{S}'\cap [1,y]} \sum_{q\in \mathcal{Q}_{H}\setminus \bm{\mathcal{T}}'_{H}} \bm{\lambda}'(H; q, n-\alpha_{q,i}-qh)\ll \frac{y \#\mathcal{Q}_{H}}{H^{M-4 - 3\varepsilon}}.
\]By summing over $h\leq KH$, we get
\[
\mathbb{E}\sum_{n\in \mathbf{S}'\cap [1,y]} \sum_{q\in \mathcal{Q}_{H}\setminus \bm{\mathcal{T}}'_{H}}
 \sum_{h\leq KH}\bm{\lambda}'(H; q, n-\alpha_{q,i}-qh)\ll \frac{y \#\mathcal{Q}_{H}}{H^{M-5 - 3\varepsilon}}.
\]We put
\begin{equation}\label{T.1.DEF}
\bm{\mathcal{W}}'_{H,i}=\bigg\{n\in \mathbf{S}'\cap [1, y]:
\sum_{q\in \mathcal{Q}_{H}\setminus \bm{\mathcal{T}}'_{H}}
 \sum_{h\leq KH}\bm{\lambda}'(H; q, n-\alpha_{q,i}-qh)\geq \frac{ \#\mathcal{Q}_{H}[KH]}{H^{1+\varepsilon}\sigma_2}\bigg\}.
\end{equation}We have
\begin{align*}
\mathbb{E} \#\bm{\mathcal{W}}'_{H,i}&=
\mathbb{E}\sum_{n\in \bm{\mathcal{W}}'_{H,i}} \frac{\sum_{q\in \mathcal{Q}_{H}\setminus \bm{\mathcal{T}}'_{H}}
 \sum_{h\leq KH}\bm{\lambda}'(H; q, n-\alpha_{q,i}-qh)}{\sum_{q\in \mathcal{Q}_{H}\setminus \bm{\mathcal{T}}'_{H}}
 \sum_{h\leq KH}\bm{\lambda}'(H; q, n-\alpha_{q,i}-qh)}\\
 &\leq \frac{H^{1+\varepsilon}\sigma_2}{\#\mathcal{Q}_{H}[KH]}
 \mathbb{E}\sum_{n\in \mathbf{S}'\cap [1,y]}\sum_{q\in \mathcal{Q}_{H}\setminus \bm{\mathcal{T}}'_{H}}
 \sum_{h\leq KH}\bm{\lambda}'(H; q, n-\alpha_{q,i}-qh)\\
 &\ll \frac{y\sigma_2}{H^{M-5-4\varepsilon}} \ll \sigma y \frac{\log H}{H^{M-5 - 4\varepsilon}}\ll \frac{\sigma y}{H^{M-5-5\varepsilon}}.
\end{align*}By Markov's inequality, we have
\begin{equation}\label{MARKOV.2}
\mathbb{P}\bigg(\#\bm{\mathcal{W}}'_{H,i} \leq \frac{\sigma y}{H^{1+\varepsilon}}\bigg)=1-O\Big(\frac{1}{H^{M-6-6\varepsilon}}\Big).
\end{equation} We see from \eqref{L2:H_alpha}, \eqref{MARKOV.1} and \eqref{MARKOV.2} that with probability $1 - O\big((\log x)^{-\delta \gamma}\big)$, where
\[
\gamma = \min (1-\varepsilon, M-6-6\varepsilon)  >0,
\] the relations
\[
\#\bm{\mathcal{E}}'_{H,i} \leq \frac{\sigma y}{H^{1+\varepsilon}},\qquad
\#\bm{\mathcal{W}}'_{H,i} \leq \frac{\sigma y}{H^{1+\varepsilon}}
\]hold for every $H\in \mathfrak{H}'$ and $i\in\{1,2\}$ simultaneously.

Fix $H\in \mathfrak{H}''$ and $i\in\{1,2\}$. From \eqref{L3:lamda.AP.2} we have
\begin{align*}
\mathbb{E}&\sum_{n\in \mathbf{S}''\cap [-y,-1]} \Bigg( \sum_{q\in \mathcal{Q}_{H}}\sum_{h\leq KH}\bm{\lambda}''(H;q, n-\alpha_{q, i}+qh) - \frac{\#\mathcal{Q}_{H} [KH]}{\sigma_{2}}\Bigg)^{2}\\
&\ll \frac{\log H}{H^{M-2}} \bigg(\frac{\#\mathcal{Q}_{H}[KH]}{\sigma_2}\bigg)^{2}\sigma y
\ll \frac{1}{H^{M-2 - \varepsilon}} \bigg(\frac{\#\mathcal{Q}_{H}[KH]}{\sigma_2}\bigg)^{2}\sigma y.
\end{align*}We put
\begin{align}
\bm{\mathcal{E}}''_{H, i}=\Bigg\{&n\in \mathbf{S}''\cap [-y,-1]:\notag\\
 &\Bigg|\sum_{q\in \mathcal{Q}_{H}}\sum_{h\leq K H}\bm{\lambda}''(H;q, n-\alpha_{q, i}+qh) - \frac{\#\mathcal{Q}_{H} [KH]}{\sigma_{2}}\Bigg|\geq \frac{\#\mathcal{Q}_{H} [KH]}{\sigma_{2} H^{1+ \varepsilon}}\Bigg\}.\label{E.2.DEF}
\end{align}Arguing as above, we obtain
\[
\mathbb{E} \#\bm{\mathcal{E}}''_{H, i}
\ll \frac{\sigma y}{H^{M-4-3\varepsilon}}\ll \frac{\sigma y}{H^{2}}.
\]By Markov's inequality, we have
\begin{equation}\label{MARKOV.3}
\mathbb{P}\Big(\# \bm{\mathcal{E}}''_{H,i}\leq \frac{\sigma y}{H^{1+\varepsilon}}\Big)= 1-O\Big(\frac{1}{H^{1- \varepsilon}}\Big).
\end{equation}

We next estimate the contribution from ``bad'' primes $q\in \mathcal{Q}_{H}\setminus \bm{\mathcal{T}}''_{H}$. For any $h\leq KH$, by the Cauchy-Schwarz inequality we have
\begin{align}
\mathbb{E}\sum_{n\in \mathbf{S}''\cap [-y,-1]} &\sum_{q\in \mathcal{Q}_{H}\setminus \bm{\mathcal{T}}''_{H}} \bm{\lambda}''(H; q, n-\alpha_{q,i}+qh)\leq \big(\mathbb{E}\# (\mathcal{Q}_{H}\setminus \bm{\mathcal{T}}''_{H})\big)^{1/2}\cdot\notag\\
&\cdot\Bigg(\mathbb{E} \sum_{q\in \mathcal{Q}_{H}\setminus \bm{\mathcal{T}}''_{H}} \bigg(\sum_{-y\leq n\leq -1} \bm{\lambda}''(H; q, n-\alpha_{q, i}+qh)\bigg)^{2}\Bigg)^{1/2}.\label{CAUSHY.LAMBDA.2}
\end{align}Given $q\in \mathcal{Q}_{H}\setminus \bm{\mathcal{T}}''_{H}$, we have
\[
\bigg|\sum_{-y\leq n\leq -1} \bm{\lambda}''(H; q, n-\alpha_{q,i}+qh)\bigg|\leq \bigg|\sum_{-y\leq n\leq -1} \bm{\lambda}''(H; q, n-\alpha_{q, i}+qh)- (K+2)y\bigg| + (K+2)y.
\] Applying \eqref{n.qh.range.2}, we obtain
\begin{align*}
\bigg|\sum_{-y\leq n\leq -1} \bm{\lambda}''(H; q, n&-\alpha_{q, i}+qh)- (K+2)y\bigg|\\
 &\leq\max\Bigg((K+2)y, \bigg|\sum_{-y \leq n< (K+1)y} \bm{\lambda}''(H; q, n)- (K+2)y\bigg|\Bigg)\\
 &\leq\bigg|\sum_{-y \leq n < (K+1)y} \bm{\lambda}''(H; q, n)- (K+2)y\bigg|+ (K+2)y.
\end{align*}

Since $(a+b)^{2}\leq 2 (a^{2}+b^{2})$, we get
\[
\bigg(\sum_{-y\leq n\leq -1} \bm{\lambda}''(H; q, n-\alpha_{q,i}+qh)\bigg)^{2}\leq 2\bigg(\sum_{-y \leq n< (K+1)y} \bm{\lambda}''(H; q, n)- (K+2)y\bigg)^{2}+ 8(K+2)^{2}y^{2}.
\]Applying \eqref{SUMMA.LAMBDA.2.M} and \eqref{STRIKE.2}, we obtain
\begin{align*}
\mathbb{E} \sum_{q\in \mathcal{Q}_{H,\nu}\setminus \bm{\mathcal{T}}''_{H}} &\bigg(\sum_{-y\leq n\leq -1} \bm{\lambda}''(H; q, n-\alpha_{q,i}+qh)\bigg)^{2}\\
 &\leq2\mathbb{E} \sum_{q\in \mathcal{Q}_{H}\setminus \bm{\mathcal{T}}''_{H}} \bigg(\sum_{-y\leq n< (K+1) y} \bm{\lambda}''(H; q, n)- (K+2)y\bigg)^{2}\\
 &+ 8\big((K+2)y\big)^{2}\mathbb{E}\# (\mathcal{Q}_{H}\setminus \bm{\mathcal{T}}''_{H})
 \ll \frac{y^{2}\#\mathcal{Q}_{H}}{H^{M-2 -\varepsilon}} + \frac{y^{2}\#\mathcal{Q}_{H}}{H^{M-4 - 3\varepsilon}}\ll
 \frac{y^{2}\#\mathcal{Q}_{H}}{H^{M-4-3\varepsilon}}.
\end{align*} By \eqref{CAUSHY.LAMBDA.2} and \eqref{STRIKE.2}, we have
\[
\mathbb{E}\sum_{n\in \mathbf{S}''\cap [-y,-1]} \sum_{q\in \mathcal{Q}_{H}\setminus \bm{\mathcal{T}}''_{H}} \bm{\lambda}''(H; q, n-\alpha_{q,i}+qh)\ll \frac{y \#\mathcal{Q}_{H}}{H^{M-4 - 3\varepsilon}}.
\]By summing over $h\leq KH$, we get
\[
\mathbb{E}\sum_{n\in \mathbf{S}''\cap [-y,-1]} \sum_{q\in \mathcal{Q}_{H}\setminus \bm{\mathcal{T}}''_{H}}
 \sum_{h\leq KH}\bm{\lambda}''(H; q, n-\alpha_{q,i}+qh)\ll \frac{y \#\mathcal{Q}_{H}}{H^{M-5 - 3\varepsilon}}.
\]We put
\begin{equation}\label{T.2.DEF}
\bm{\mathcal{W}}''_{H,i}=\bigg\{n\in \mathbf{S}''\cap [-y, -1]:
\sum_{q\in \mathcal{Q}_{H}\setminus \bm{\mathcal{T}}''_{H}}
 \sum_{h\leq KH}\bm{\lambda}''(H; q, n-\alpha_{q,i}+qh)\geq \frac{ \#\mathcal{Q}_{H}[KH]}{H^{1+\varepsilon}\sigma_2}\bigg\}.
\end{equation}We see that
\[
\mathbb{E} \#\bm{\mathcal{W}}''_{H,i}\ll \frac{y\sigma_2}{H^{M-5-4\varepsilon}} \ll \sigma y \frac{\log H}{H^{M-5 - 4\varepsilon}}\ll \frac{\sigma y}{H^{M-5-5\varepsilon}}.
\]By Markov's inequality, we have
\begin{equation}\label{MARKOV.4}
\mathbb{P}\bigg(\#\bm{\mathcal{W}}''_{H,i} \leq \frac{\sigma y}{H^{1+\varepsilon}}\bigg)=1-O\Big(\frac{1}{H^{M-6-6\varepsilon}}\Big).
\end{equation} We see from \eqref{L2:H_alpha}, \eqref{MARKOV.3} and \eqref{MARKOV.4} that with probability $1 - O\big((\log x)^{-\delta \gamma}\big)$, where
\[
\gamma = \min (1-\varepsilon, M-6-6\varepsilon)  >0,
\] the relations
\[
\#\bm{\mathcal{E}}''_{H,i} \leq \frac{\sigma y}{H^{1+\varepsilon}},\qquad
\#\bm{\mathcal{W}}''_{H,i} \leq \frac{\sigma y}{H^{1+\varepsilon}}
\]hold for every $H\in \mathfrak{H}''$ and $i\in\{1,2\}$ simultaneously.

Now we make a choice of $d$ (mod $P(z)$). We consider the event that $\# (\mathbf{S}' \cap [1,y])\leq 2\sigma y$ and that $\# (\mathbf{S}'' \cap [-y,-1])\leq 2\sigma y$, that for each $H\in \mathfrak{H}'$,  $i\in\{1,2\}$, the sets $\bm{\mathcal{E}}'_{H,i}$, $\bm{\mathcal{W}}'_{H,i}$ have size at most $(\sigma y) H^{-1-\varepsilon}$, that for each $H\in \mathfrak{H}''$, $i\in\{1,2\}$, the sets $\bm{\mathcal{E}}''_{H,i}$, $\bm{\mathcal{W}}''_{H,i}$ have size at most $(\sigma y) H^{-1-\varepsilon}$, that for each $H\in\mathfrak{H}'$ the relation \eqref{BASIC.I} holds, and that for each $H\in\mathfrak{H}''$ the relation \eqref{BASIC.II} holds. By the above discussion, this event holds with probability $1- o(1)$ as $x\to \infty$, and so this probability is at least $1/2$ provided that $x$ is large enough in terms of $\delta$, $M$, $K$, $\xi$, and $\varepsilon$. From now, we fix a $d$ mod $P(z)$ such that it is so, and thus all of our random sets and weights become deterministic. We see that,  $\mathcal{T}'$ and $\mathcal{T}''$ are non-empty subsets of $\mathcal{Q}'$ and $\mathcal{Q}''$ respectively, and thus, by \eqref{SUM.LAMBDA.1} and \eqref{SUM.LAMBDA.2}, parts \textup{(i)} and \textup{(ii)} of Proposition \ref{P1} are verified. Now we verify parts \textup{(iii)} and \textup{(iv)} of Proposition \ref{P1}.

Fix  $i\in\{1,2\}$. We set
\[
V'_{i} = (S'\cap [1,y]) \setminus \bigcup_{H\in \mathfrak{H}'}
(\mathcal{E}'_{H,i}\cup \mathcal{W}'_{H,i}).
\]By \eqref{L2:H_alpha}, we have
\begin{align*}
\#\Bigg(\bigcup_{H\in \mathfrak{H}'}
(\mathcal{E}'_{H,i}\cup \mathcal{W}'_{H,i})\Bigg)\ll \sigma y
\sum_{H\in \mathfrak{H}'}\frac{1}{H^{1+\varepsilon}}&\ll \frac{\sigma y}{(\log x)^{\delta (1+\varepsilon)}}\\
&\ll \frac{x}{(\log x)^{1+\delta \varepsilon}}< \frac{x}{100 \log x},
\end{align*}verifying \eqref{V.1.DOPOLNENIYE}.

We fix arbitrary $n\in V'_{i}$. For such $n$, the inequalities  \eqref{E.1.DEF} and \eqref{T.1.DEF} both fail, and therefore for each $H\in \mathfrak{H}'$,
\begin{align*}
\sum_{q\in \mathcal{T}'_{H}}\sum_{h\leq K H}\lambda '(H; q, n-\alpha_{q,i}-qh)&=
\bigg(1+ O\bigg(\frac{1}{H^{1+\varepsilon}}\bigg)\bigg) \frac{\#\mathcal{Q}_{H}[KH]}{\sigma_2}\\
&=\bigg(1+ O\bigg(\frac{1}{(\log x)^{\delta(1+\varepsilon)}}\bigg)\bigg) \frac{\#\mathcal{Q}_{H}[KH]}{\sigma_2}.
\end{align*}Summing over all $H\in \mathfrak{H}'$, we have
\begin{align*}
\sum_{q\in \mathcal{T}'}\sum_{h\leq KH_{q}}\lambda '(H; q, n-\alpha_{q, i}-qh)&=
\sum_{H\in \mathfrak{H}'}\sum_{q\in \mathcal{T}'_{H}}\sum_{h\leq KH_{q}}\lambda'(H; q, n-\alpha_{q, i}-qh)\\
&=\bigg(1+O\bigg(\frac{1}{(\log x)^{\delta (1+\varepsilon)}}\bigg)\bigg)C' (K+2)y,
\end{align*}where
\[
C'= \frac{1}{(K+2)y}\sum_{H\in \mathfrak{H}'}\frac{\# \mathcal{Q}_{H} [KH]}{\sigma_{2}}.
\] Note that $C'$ depends on $x$, $K$, $M$, $\xi$, and $\delta$, but not on $n$ and $i$. Since
\[
[KH]=KH\Big(1+ O\Big(\frac{1}{H}\Big)\Big) = KH\Big(1+ O\Big(\frac{1}{(\log x)^{\delta}}\Big)\Big),
\] we get, using \eqref{QH.NU.ASYMPT} and \eqref{SIGMA.2},
\[
C'\sim \frac{K (1-1/\xi)}{2(K+2)M} \sum_{H\in\mathfrak{H}'}\frac{1}{\log H}.
\]Recalling the definition of $\mathfrak{H}'$, we see that
\[
C'\sim \frac{K (1-1/\xi)}{4(K+2)M \log \xi} \sum_{j} \frac{1}{j},
\]where $j$ runs over the interval
\[
\frac{\delta \log\log x}{2 \log \xi}(1+o(1)) \leq j \leq \frac{\log\log x}{4\log \xi} (1+o(1)).
\]We thus obtain
\[
C'\sim \frac{K}{4(K+2) M} \frac{1-1/\xi}{\log \xi} \log \Big(\frac{1}{2\delta}\Big).
\]

Similarly, for fixed $i\in\{1,2\}$, we set
\[
V''_{i} = (S''\cap [-y,-1]) \setminus \bigcup_{H\in \mathfrak{H}''}
(\mathcal{E}''_{H,i}\cup \mathcal{W}''_{H,i}).
\]By \eqref{L2:H_alpha}, we have
\begin{align*}
\#\Bigg(\bigcup_{H\in \mathfrak{H}''}
(\mathcal{E}''_{H,i}\cup \mathcal{W}''_{H,i})\Bigg)\ll \sigma y
\sum_{H\in \mathfrak{H}''}\frac{1}{H^{1+\varepsilon}}&\ll \frac{\sigma y}{(\log x)^{\delta (1+\varepsilon)}}\\
&\ll \frac{x}{(\log x)^{1+\delta \varepsilon}}< \frac{ x}{100 \log x},
\end{align*}verifying \eqref{V.2.DOPOLNENIYE}.

We fix arbitrary $n\in V''_{i}$. For such $n$, the inequalities  \eqref{E.2.DEF} and \eqref{T.2.DEF} both fail, and therefore for each $H\in \mathfrak{H}''$,
\begin{align*}
\sum_{q\in \mathcal{T}''_{H}}\sum_{h\leq K H}\lambda ''(H; q, n-\alpha_{q,i}+qh)&=
\bigg(1+ O\bigg(\frac{1}{H^{1+\varepsilon}}\bigg)\bigg) \frac{\#\mathcal{Q}_{H}[KH]}{\sigma_2}\\
&=\bigg(1+ O\bigg(\frac{1}{(\log x)^{\delta(1+\varepsilon)}}\bigg)\bigg) \frac{\#\mathcal{Q}_{H}[KH]}{\sigma_2}.
\end{align*}Summing over all $H\in \mathfrak{H}''$, we have
\begin{align*}
\sum_{q\in \mathcal{T}''}\sum_{h\leq KH_{q}}\lambda ''(H; q, n-\alpha_{q, i}+qh)&=
\sum_{H\in \mathfrak{H}''}\sum_{q\in \mathcal{T}''_{H}}\sum_{h\leq KH_{q}}\lambda''(H; q, n-\alpha_{q, i}+qh)\\
&=\bigg(1+O\bigg(\frac{1}{(\log x)^{\delta (1+\varepsilon)}}\bigg)\bigg)C'' (K+2)y,
\end{align*}where
\[
C''= \frac{1}{(K+2)y}\sum_{H\in \mathfrak{H}''}\frac{\# \mathcal{Q}_{H} [KH]}{\sigma_{2}}.
\] Note that $C''$ depends on $x$, $K$, $M$, $\xi$, and $\delta$, but not on $n$ and $i$. Arguing as above, we obtain
\[
C''\sim \frac{K}{4(K+2) M} \frac{1-1/\xi}{\log \xi} \log \Big(\frac{1}{2\delta}\Big).
\]

Since $\delta < C(1/2)$, we see from \eqref{C.RHO.INEQUALITY} that
\[
C'\geq \frac{10^{2\delta}}{2}  \quad\text{and}\quad C'' \geq \frac{10^{2\delta}}{2}
 \] provided that $(M-6)$ and $(\xi - 1)$ are sufficiently small in terms of $\delta$, $K$ is sufficiently large in terms of $\delta$, $0 < \varepsilon < (M-6)/7$, and $x$ is sufficiently large in terms of $\delta$, $M$, $\xi$, $K$, $\varepsilon$. Also $C' \leq 50$ and $C'' \leq 50$  due to $\delta > 10^{-6}$. So parts \textup{(iii)} and \textup{(iv)} of Proposition \ref{P1} are verified.

 This completes the proof of Proposition \ref{P1} assuming Proposition \ref{P2}. Thus we are left to prove Proposition \ref{P2}. This will be done in the next sections.

 \section{Preparatory Lemmas}

 For $H\in \mathfrak{H}$, let $\mathcal{D}_{H}$ be the collection of square-free numbers $n\in <P_{3}>$, all of whose prime divisors lie in $(H^{M}, z]$. We may assume that $H^{M}> a+|b|$. For each $n\in \mathcal{D}_{H}$, let $I_{n}\subset \mathbb{Z}/ n \mathbb{Z}$ denote the collection of residue classes $a\,\textup{(mod $n$)}$ such that $a\,\textup{(mod $q$)}\in I_{q}$ for all $q|n$. Further, for $A>0$, let
 \begin{equation}\label{E.DEF}
 E_{A}(m; H) = \big(1_{m\neq 0}\big)\sum_{n\in \mathcal{D}_{H}\backslash \{1\}}\frac{A^{\nu (n)}}{n} 1_{m\,(\textup{mod}\,n) \in I_{n} - I_{n}}.
 \end{equation} We note that $E_{A}(m;H)$ is increasing in $A$ and $E_{A}(m; H)= E_{A}(-m; H)$ for all $m\in \mathbb{Z}$. For $n\in \mathcal{D}_{H}\backslash\{1\}$, we see that $m\textup{(mod $n$)}\in I_{n}-I_{n}$ if and only if $m\,\textup{(mod $q$)}\in I_{q}-I_{q}$ for all $q|n$.

 \begin{lemma}\label{L1.PREP}
 Let $10 < H < z ^{1/M}$, $2\leq l \leq 10 KH$, $\mathcal{U}\subset \mathcal{V}$ be two finite sets of integers with $\# \mathcal{V} = l$. Then
 \[
 \mathbb{P} (\mathcal{U} \subset \mathbf{S}'_{2})= \mathbb{P} (\mathcal{U} \subset \mathbf{S}''_{2}) = \sigma_{2}^{\#\mathcal{U}}
 \bigg(1+O \bigg(\frac{\# \mathcal{U}}{H^{M}}\bigg)+ O\bigg(\frac{1}{l^{2}} \sum_{v, v' \in \mathcal{V}}E_{3l^{2}}(v-v'; H)\bigg)\bigg).
 \]
 \end{lemma}
 \begin{proof}
 For each prime $q\in (H^{M}, z]$, let $\mathbf{d}_{2,q}\in \mathbb{Z}/ q \mathbb{Z}$ be the reduction of $\mathbf{d}_{2}$ modulo $q$, thus each $\mathbf{d}_{2,q}$ is uniformly distributed in $\mathbb{Z}/q\mathbb{Z}$ and $\mathbf{d}_{2,q}$ are independent in $q$. Let $\mathcal{U}_{q}$ denote the set of residue classes $\mathcal{U}$ (mod $q$). By the Chinese Remainder Theorem, we have
 \begin{align*}
 \mathbb{P}(\mathcal{U}\subset \mathbf{S}'_{2})&=\prod_{q\in (H^{M}, z]} \mathbb{P}(\mathcal{U}_{q}\cap (\mathbf{d}_{2,q} + I_{q})= \emptyset)\\
 &=\prod_{q\in (H^{M}, z]} (1- \mathbb{P}(\mathbf{d}_{2,q}\in \mathcal{U}_{q} - I_{q}))\\
 &=\prod_{q\in (H^{M}, z]} \bigg(1- \frac{\#(\mathcal{U}_{q}-I_{q})}{q}\bigg).
 \end{align*} Arguing as above, we obtain

 \begin{align*}
 \mathbb{P}(\mathcal{U}\subset \mathbf{S}''_{2})&=\prod_{q\in (H^{M}, z]} \mathbb{P}(\mathcal{U}_{q}\cap ((-N-\mathbf{d}_{2,q}) + I_{q})= \emptyset)\\
 &=\prod_{q\in (H^{M}, z]} (1- \mathbb{P}(\mathbf{d}_{2,q}\in -N+ (I_{q}- \mathcal{U}_{q})))\\
 &=\prod_{q\in (H^{M}, z]} \bigg(1- \frac{\#(-N+(I_{q}- \mathcal{U}_{q}))}{q}\bigg).
 \end{align*}Since
 \[
 \# (-N + (I_{q}-\mathcal{U}_{q}))=\# (I_{q}-\mathcal{U}_{q})=\# (\mathcal{U}_{q}-I_{q}),
 \]we obtain
 \[
 \mathbb{P}(\mathcal{U}\subset \mathbf{S}''_{2}) = \mathbb{P}(\mathcal{U}\subset \mathbf{S}'_{2})=
 \prod_{q\in (H^{M}, z]} \bigg(1- \frac{\#(\mathcal{U}_{q}-I_{q})}{q}\bigg).
 \]

 Let $k=\# \mathcal{U}$. We may crudely estimate the size of the difference set $\mathcal{U}_{q} - I_{q}$ by
 \[
 k\# I_{q} \geq \#(\mathcal{U}_{q} - I_{q})\geq k\# I_{q} - \# I_{q} \sum_{u, u' \in \mathcal{U}, u\neq u'}
 1_{u-u'\, \textup{(mod $q$)} \in I_{q}-I_{q}}.
 \]Since $\# I_{q} = 2$ and $k \leq 10 KH$, we have
 \[
 2k\geq \#(\mathcal{U}_{q} - I_{q})\geq 2k - 2 \sum_{u, u' \in \mathcal{U}, u\neq u'}
 1_{u-u'\, \textup{(mod $q$)} \in I_{q}-I_{q}}
 \]and $2k \leq 20 KH < q/10$ for $x$ large enough. Thus,
 \[
 \bigg(1- \frac{\#(\mathcal{U}_{q} - I_{q})}{q}\bigg) =\bigg(1- \frac{2k}{q}\bigg)
 \bigg(1+ \frac{2k - \#(\mathcal{U}_{q}-I_{q})}{q-2k}\bigg) =
 \bigg(1- \frac{2k}{q}\bigg) \Delta_{q}.
 \]Since $e^{x} \geq 1 + x$, we have
 \begin{align*}
 1\leq \Delta_{q} &\leq 1 + \frac{3}{q} \sum_{u, u' \in \mathcal{U}, u\neq u'}
 1_{u-u'\, \textup{(mod $q$)} \in I_{q}-I_{q}}\\
 &\leq \prod_{u, u' \in \mathcal{U}, u\neq u'}\textup{exp}\bigg(3 \frac{1_{u-u'\, \textup{(mod $q$)} \in I_{q}-I_{q}}}{q} \bigg)\\
 &\leq \prod_{v, v' \in \mathcal{V}, v\neq v'}\textup{exp}\bigg(3 \frac{1_{v-v'\, \textup{(mod $q$)} \in I_{q}-I_{q}}}{q} \bigg).
 \end{align*}Here we have enlarged the summation over pairs of numbers from $\mathcal{V}$.

 By the arithmetic mean - geometric mean inequality, we have
 \begin{align*}
 \prod_{q\in (H^{M}, z]} \Delta_{q}&\leq \prod_{v, v' \in \mathcal{V}, v\neq v'}\prod_{q\in (H^{M}, z]}
 \textup{exp}\bigg(3 \frac{1_{v-v'\, \textup{(mod $q$)} \in I_{q}-I_{q}}}{q} \bigg)\\
 &\leq \frac{2}{l^{2}-l} \sum_{v, v' \in \mathcal{V}, v\neq v'}
 \prod_{q\in (H^{M}, z]}
 \textup{exp}\bigg(3\bigg(\frac{l^{2}-l}{2}\bigg) \frac{1_{v-v'\, \textup{(mod $q$)} \in I_{q}-I_{q}}}{q} \bigg)\\
 &\leq \frac{2}{l^{2}-l} \sum_{v, v' \in \mathcal{V}, v\neq v'}
 \prod_{q\in (H^{M}, z]}
 \textup{exp}\bigg(3\bigg(\frac{l^{2}}{2}\bigg) \frac{1_{v-v'\, \textup{(mod $q$)} \in I_{q}-I_{q}}}{q} \bigg).
 \end{align*} Since $l \leq 10 KH$ and $q> H^{M}> H^6$,
 by \eqref{H_range} we have
 \[
 0< \frac{3 l^{2}}{2q}\leq \frac{150 K^{2}}{H^{4}}\leq \frac{150 K^{2}}{(\log x)^{4\delta}}< \frac{1}{2}
 \]for $x$ large enough. We observe that $e^{x} \leq 1+ 2x$ for $x\in [0, 1/2]$. Hence,
 \begin{align*}
 \prod_{q\in (H^{M}, z]} \Delta_{q} &\leq \frac{2}{l^{2}-l} \sum_{v, v' \in \mathcal{V}, v\neq v'}
 \prod_{q\in (H^{M}, z]}
 \bigg(1+ 3l^{2}\frac{1_{v-v'\, \textup{(mod $q$)} \in I_{q}-I_{q}}}{q} \bigg)\\
 &=\frac{2}{l^{2}-l} \sum_{v, v' \in \mathcal{V}, v\neq v'}(1 + E_{3l^{2}}(v-v'; H))\\
 &= 1 + \frac{2}{l^{2}-l} \sum_{v, v' \in \mathcal{V}} E_{3l^2}(v-v'; H).
 \end{align*}

 It is easy to see that
 \[
 \prod_{H^{M}< q \leq z}\bigg(1-\frac{2k}{q}\bigg) = \prod_{H^{M}<q \leq z} \bigg(1-\frac{2}{q}\bigg)^{k}
 \bigg(1 + O \bigg(\frac{k^{2}}{H^{M}}\bigg)\bigg)=\sigma_{2}^{k}\bigg(1 + O \bigg(\frac{k^{2}}{H^{M}}\bigg)\bigg).
 \]Since $k \leq 10 KH$ and $l\geq 2$, we have $O(k^{2}/H^{M}) = O(1)$ and $l^{2} - l \geq l^{2}/2$. We obtain
 \begin{align*}
 \mathbb{P}(\mathcal{U}\subset \mathbf{S}'_{2}) = \mathbb{P}(\mathcal{U}\subset \mathbf{S}''_{2})&=
 \sigma_{2}^{k}\bigg(1 + O \bigg(\frac{k^{2}}{H^{M}}\bigg)\bigg) \bigg(1 + \frac{2}{l^{2}-l} \sum_{v, v' \in \mathcal{V}} E_{3l^2}(v-v'; H)\bigg)\\
 &=\sigma_{2}^{k}\bigg(1 + O \bigg(\frac{k^{2}}{H^{M}}\bigg)+ O\bigg(\frac{1}{l^{2}}\sum_{v, v' \in \mathcal{V}} E_{3l^2}(v-v'; H)\bigg) \bigg).
 \end{align*}
 \end{proof}

 \begin{lemma}\label{L2:PREP}
 Suppose that $10 < H < z^{1/M}$, and that $\{m_{t}\}_{t\in T}$ is a sequence of integers indexed by a finite set $T$, obeying the bounds
 \begin{equation}\label{m.t.DISTRIBUTION}
 \#\{ t\in T: m_{t}\equiv a\textup{ (mod $n$)}\}\ll \frac{X}{\varphi (n)} + R
 \end{equation}for some $X, R >0$ and all $n\in \mathcal{D}_{H}\backslash \{1\}$ and $a\in \mathbb{Z}/ n \mathbb{Z}$. Then, for any $A$ satisfying $1 \leq A \leq H^{M}$ and any integer $j$, we have
 \[
 \sum_{t\in T} E_{A}(m_{t}+ j; H) \ll X \frac{A}{H^{M}} + R\,\textup{exp} (3 A \log\log y).
 \]
 \end{lemma}
 \begin{proof}By the Chinese Remainder Theorem, for any $n\in \mathcal{D}_{H}$ we have
 \[
 \# I_{n} = \prod_{q|n}\# I_{q} = 2^{\nu (n)}.
 \]Hence,
 \[
 \# (I_{n} - I_{n})\leq 4^{\nu (n)}.
 \]From \eqref{E.DEF} and \eqref{m.t.DISTRIBUTION} we have
 \begin{align*}
 \sum_{t\in T} E_{A}(m_{t}+ j; H)&=\sum_{n\in \mathcal{D}_{H}\backslash \{1\}} \frac{A^{\nu (n)}}{n}
 \sum_{a\in I_{n}-I_{n}} \#\{t\in T: m_{t}+j\neq 0, m_{t}+j\equiv a\textup{ (mod $n$)}\}\\
 & \ll \sum_{n\in \mathcal{D}_{H}\backslash \{1\}} \frac{(4 A)^{\nu (n)}}{n}\bigg(\frac{X}{\varphi (n)}+ R\bigg).
 \end{align*}Using Euler products and the inequality $1+x \leq e^{x}$,  we have
 \begin{align*}
 \sum_{n\in \mathcal{D}_{H}} \frac{(4A)^{\nu (n)}}{n\varphi (n)}&=
 \prod_{q\in (H^{M}, z]}\bigg(1 + \frac{4A}{q^{2}- q}\bigg)\leq
 \prod_{q\in (H^{M}, z]}\bigg(1 + \frac{8A}{q^{2}}\bigg)\\
 &\leq \textup{exp}\bigg(8A \sum_{H^{M}< q \leq z} \frac{1}{q^{2}} \bigg)\leq
 \textup{exp}\bigg(\frac{16A}{H^{M}}\bigg)= 1 + O\bigg(\frac{A}{H^{M}}\bigg).
 \end{align*}Since
 \[
 \sum_{q \leq z}\frac{1}{q} \sim \frac{1}{2} \log\log z\leq \frac{3}{4}\log\log y,
 \]we have
 \[
 \sum_{n\in \mathcal{D}_{H}} \frac{(4A)^{\nu (n)}}{n} = \prod_{q\in (H^{M}, z]}\bigg(1 + \frac{4A}{q}\bigg)
 \leq \textup{exp}\bigg(4A \sum_{H^{M}< q\leq z}\frac{1}{q} \bigg)
 \leq \textup{exp}(3A \log \log y).
 \]

 \end{proof}

 For non-zero $v$, define
 \begin{equation}\label{N.DEF}
 N(v)=\#\{q: v\textup{ (mod $q$)}\in I_{q} - I_{q}\}.
 \end{equation}
 \begin{lemma}\label{L3:Prep}
  Let $u \geq 10$, $w$ and $k$ are integers with $|w|\leq u$ and $k\geq 1$. Suppose that $\Omega$ is a subset of $\mathbb{N}$, and for each prime $q\in \Omega$, let $m_{q}$ be a non-zero integer with $m_{q}\textup{ (mod $q$)}\in I_{q}-I_{q}$ and $|m_{q}|\leq u$. Then
 \[
 \#\{q\in \Omega: m_{q}+w \neq 0, N(m_{q}+w) = k\}\ll 2^{-k/3}u(\log u)^{2}.
 \]
 \end{lemma}
 \begin{proof}Let us denote
 \[
 \Omega_{1} = \{q\in \Omega: m_{q}+w\neq 0, N(m_{q}+w)=k\},\qquad T=\#\Omega_{1}.
 \]

Let
\[
F(n) = n (an+b)(-an+b)=F_{1}(n)F_{2}(n)F_{3}(n).
\]Since $a\nmid b$, the linear form $ax+b \neq 0$ for all $x\in \mathbb{Z}$. Therefore $F(n)\neq 0$ for $n\neq 0$. Also, if $a\textup{ (mod $q$)}\in I_{q}- I_{q}$, then $F(a)\equiv 0$ (mod $q$).

 If $q\in \Omega_{1}$ and $m_{q}+w = m$, then $0<|m|\leq 2 u$, $m\neq w$, and $N(m)=k$. We observe that if $m$ satisfies $0< |m|\leq 2 u$ and $m\neq w$, then there are $O(\log u)$ primes dividing $F(m-w)$.  Thus, for any $m$ satisfying $0< |m|\leq 2 u$, $m\neq w$, and $N(m)=k$, there are $O(\log u)$ primes $q\in \Omega_{1}$ such that $m_{q}+w = m$. We obtain
\begin{align*}
T &\ll \#\{m: 0<|m|\leq 2u, m\neq w, N(m)=k\}\log u\\
&\leq \#\{m: 0<|m|\leq 2u, N(m)=k\}\log u = T_{1}\log u.
\end{align*}

If $m$ is an integer such that $0<|m|\leq 2u$ and  $N(m)=k$, then there are $k$ distinct primes $q$ dividing $F(m)$, and therefore, for some $i\in\{1, 2,3\}$, at least $k/3$ distinct primes $q$ dividing $F_{i}(m)$. Hence,
\begin{align*}
T_{1}&\leq \sum_{i=1}^{3}\#\Big\{m: 0<|m|\leq 2 u, \nu (|F_{i}(m)|)\geq \frac{k}{3}\Big\}\\
&\leq \sum_{i=1}^{3}\#\Big\{m: 0<|m|\leq 2 u, \tau (|F_{i}(m)|)\geq 2^{k/3}\Big\}\\
&\leq \sum_{i=1}^{3} 2^{-k/3}\sum_{0< |m|\leq 2u}\tau (|F_{i}(m)|)\ll 2^{-k/3} u \log u.
\end{align*}We obtain $T \ll 2^{-k/3} u (\log u)^{2}$.
 \end{proof}

 \begin{lemma}\label{L4:PREP}
  Let $H\in \mathfrak{H}$. For $q\in \mathcal{Q}_{H}$, suppose that $m_{q}\textup{ (mod $q$)}\in I_{q}-I_{q}$ with $0< |m_{q}|\leq x\log x$, and suppose that $w \in \mathbb{Z}$ with $|w|\leq x\log x$. Then
 \[
 \sum_{q\in\mathcal{Q}_{H}}E_{64K^{2}H^{2}} (m_{q}+w; H)\ll \frac{\#\mathcal{Q}_{H}\log H}{H^{M-2}}.
 \]
 \end{lemma}
 \begin{proof} In this proof by $p$ we denote a prime such that $p\equiv 3$ (mod $4$) and $p>a+|b|$. Let us denote $A=64 K^2 H^2$. If $m_{q}+w\neq 0$, then
 \begin{align*}
 E_{A}(m_{q}+w; H)&\leq \prod_{\substack{m_{q}+w\textup{\,(mod $p$)}\in I_{p}-I_{p}\\ H^{M}<p\leq z}}
 \bigg(1+\frac{A}{p}\bigg)-1\\
 &\leq\textup{exp}\Bigg(A\sum_{\substack{m_{q}+w\textup{\,(mod $p$)}\in I_{p}-I_{p}\\ H^{M}<p\leq z}}\frac{1}{p}\Bigg)-1.
 \end{align*}

 Recalling the definition \eqref{N.DEF}, we see that the number of primes $p$ such that $H^{M}<p\leq z$ and $m_{q}+w\textup{ (mod $p$)}\in I_{p}-I_{p}$ is at most $N(m_{q}+w)$. Let $c_1$ be a sufficiently large constant. We put
 \begin{align*}
 \Omega_{1}&=\{q\in \mathcal{Q}_{H}: m_{q}+w\neq 0, N(m_{q}+w)\leq c_{1}\log H\},\\
 \Omega_{2}&=\{q\in \mathcal{Q}_{H}: m_{q}+w\neq 0, N(m_{q}+w)> c_{1}\log H\},
  \end{align*}and
  \[
  S_{1}=\sum_{q\in \Omega_{1}}E_{A}(m_{q}+w; H),\qquad S_{2}=\sum_{q\in \Omega_{2}}E_{A}(m_{q}+w; H).
  \]Since $E_{A}(m; H) = 0$ if $m=0$, we obtain
  \[
  \sum_{q\in\mathcal{Q}_{H}}E_{A} (m_{q}+w; H)=
  \sum_{\substack{q\in\mathcal{Q}_{H}\\ m_{q}+w \neq 0}}E_{A} (m_{q}+w; H)=S_{1}+S_{2}.
  \]

 For $q\in \Omega_{1}$, we have
 \[
 \sum_{\substack{m_{q}+w\textup{\,(mod $p$)}\in I_{p}-I_{p}\\ H^{M}<p\leq z}}\frac{1}{p}\leq \frac{c_{1}\log H}{H^{M}}.
 \]We obtain
 \[
 S_{1}\leq\# \Omega_{1}
 \bigg(\textup{exp}\Big(\frac{c_{2}\log H}{H^{M-2}}\Big)-1\bigg)\ll \frac{\# \mathcal{Q}_{H} \log H}{H^{M-2}},
 \]where $c_2 = 64 K^2 c_1$.

 Using Lemma \ref{L3:Prep} with $\Omega = \mathcal{Q}_{H}$ and $u=x\log x$ and setting $\alpha = (\log 2)/3$, we have
 \begin{align*}
 S_{2}&\leq \sum_{k>c_{1}\log H}\#\{q\in\mathcal{Q}_{H}: m_{q}+w\neq 0, N(m_{q}+w)=k\}e^{Ak/H^{M}}\\
 &\ll x (\log x)^{3}\sum_{k> c_{1}\log H}e^{-\alpha k}\textup{exp}\Big(\frac{64K^{2}k}{H^{M-2}}\Big)
 \ll x (\log x)^{3} \textup{exp}\Big(-\frac{\alpha c_{1}\log H}{2}\Big).
 \end{align*}From \eqref{SETUP:Def.y}, \eqref{H_range}, and \eqref{QH.NU.ASYMPT} we obtain
 \[
 \frac{\#\mathcal{Q}_{H}}{H^{M-2}}\gg \frac{x}{H^{M-1}(\log x)^{1-\delta}}
 \geq \frac{x}{H^{6}(\log x)^{1-\delta}}\geq \frac{x}{(\log x)^{4-\delta}}
 \geq \frac{x}{(\log x)^{4}}.
 \]Also,
 \[
 \frac{x(\log x)^{3}}{H^{(\alpha c_{1})/2}}
 \leq \frac{x(\log x)^{3}}{(\log x)^{(\delta\alpha c_{1})/2}}< \frac{x}{(\log x)^{4}},
 \]if $c_{1}$ is sufficiently large. Hence,
 \[
 S_{2}\ll \frac{\#\mathcal{Q}_{H}}{H^{M-2}} < \frac{\#\mathcal{Q}_{H}\log H}{H^{M-2}}.
 \]

 \end{proof}

 \section{The proof of part \textup{(i)} of Proposition \ref{P2}}
 It is easy to see that, for any $n\in \mathbb{Z}$,
 \[
 \mathbb{P}(n\in \mathbf{S}')=\mathbb{P}(n\in \mathbf{S}'')=\prod_{q\leq z} \bigg(\frac{q-2}{q}\bigg)=\sigma.
 \]We obtain
 \[
 \mathbb{E} \#(\mathbf{S}' \cap [1,y]) = \sum_{1\leq n \leq y} \mathbb{P}(n\in \mathbf{S}')=\sigma y
 \]and
 \[
 \mathbb{E} \#(\mathbf{S}'' \cap [-y,-1]) = \sum_{-y\leq n \leq -1} \mathbb{P}(n\in \mathbf{S}'')=\sigma y.
 \]

 Now we consider the second equation in \eqref{S.1.OSNOVA}. Here we decompose $\mathbf{S}'$ as $\mathbf{S}' = \mathbf{S}'_{1} \cap \mathbf{S}'_{2}$ using \eqref{DECOMP.1} and \eqref{DECOMP.2} with
 \[
 H=\frac{1}{4}(\log y)^{1/M}.
 \] By the Prime Number Theorem,
 \[
 P_{1}\leq \textup{exp}((1+ o(1))H^{M})\leq y ^{1/10}.
 \]By linearity of expectation,
 \begin{align*}
 \mathbb{E} \# (\mathbf{S}' \cap [1,y])^{2} &= \sum_{n_{1}, n_{2} \leq y} \mathbb{P} (n_{1}, n_{2} \in \mathbf{S}')\\
 &= \sum_{n_{1}, n_{2}\leq y}\mathbb{P}(n_{1}, n_{2}\in \mathbf{S}'_{1})\mathbb{P}(n_{1}, n_{2}\in \mathbf{S}'_{2}).
 \end{align*} Observe that the probability $\mathbb{P}(n_{1}, n_{2} \in \mathbf{S}'_{1})$ depends only on the reductions $l_1 \equiv n_{1}$ (mod $P_{1}$), $l_{2}\equiv n_{2}$ (mod $P_{1}$). We obtain
 \[
 \mathbb{E} \# (\mathbf{S}' \cap [1,y])^{2} = \sum_{1 \leq l_{1}, l_{2} \leq P_{1}} \mathbb{P}(l_{1}, l_{2} \in \mathbf{S}'_{1}) \sum_{\substack{1\leq n_{1}, n_{2}\leq y\\ n_{1} \equiv l_{1}\textup{\,(mod $P_{1}$)}\\n_{2}\equiv l_{2}\textup{\,(mod $P_{1}$)}}} \mathbb{P}(n_{1}, n_{2} \in \mathbf{S}'_{2}).
 \]

 By Lemma \ref{L1.PREP} with $\mathcal{U} = \mathcal{V} =\{n_{1}, n_{2}\}$ and $l=2$, we have
 \[
 \mathbb{P}(n_{1}, n_{2}\in \mathbf{S}'_{2}) = \sigma_{2}^{2}\Big( 1+ O\Big(\frac{1}{H^{M}}\Big)+ O\Big(E_{12}(n_{1}-n_{2}; H)\Big)\Big).
 \]We obtain
 \begin{align*}
 \mathbb{E} \# (\mathbf{S}' \cap [1,y])^{2}&=\sigma_{2}^{2}\Big(1+O\Big(\frac{1}{\log y}\Big)\Big)\sum_{1 \leq l_{1}, l_{2} \leq P_{1}} \mathbb{P}(l_{1}, l_{2} \in \mathbf{S}'_{1})
 \Big(\frac{y}{P_{1}}+ O(1)\Big)^{2}\\
 &+O\bigg(\sigma_{2}^{2}\sum_{1 \leq l_{1}, l_{2} \leq P_{1}} \mathbb{P}(l_{1}, l_{2} \in \mathbf{S}'_{1})
 \sum_{\substack{1\leq n_{1}, n_{2}\leq y\\ n_{1} \equiv l_{1}\textup{\,(mod $P_{1}$)}\\n_{2}\equiv l_{2}\textup{\,(mod $P_{1}$)}}}E_{12}(n_{1}-n_{2};H)\bigg).
 \end{align*}

 It is easy to see that
  \[
 \sum_{1 \leq l_{1}, l_{2} \leq P_{1}} \mathbb{P}(l_{1}, l_{2} \in \mathbf{S}'_{1})
 = \mathbb{E}\#(\mathbf{S}'_{1}\cap [1, P_{1}])^{2} = (\sigma_{1}P_{1})^{2},
 \]since $\#(\mathbf{S}'_{1} \cap [1, P_{1}]) = \sigma_{1}P_{1}$ always. Hence,
 \[
 \sigma_{2}^{2}\Big(1+O\Big(\frac{1}{\log y}\Big)\Big)\sum_{1 \leq l_{1}, l_{2} \leq P_{1}} \mathbb{P}(l_{1}, l_{2} \in \mathbf{S}'_{1})
 \Big(\frac{y}{P_{1}}+ O(1)\Big)^{2} = (\sigma y)^{2}\Big(1+ O\Big(\frac{1}{\log y}\Big)\Big).
 \]

 Fix $1 \leq l_{1}, \l_{2}\leq P_{1}$ and $1 \leq n_{2} \leq y$ with $n_{2}\equiv l_{2}$ (mod $P_{1}$). For any $n\in \mathcal{D}_{H}\backslash \{1\}$ and residue class $a$ (mod $n$) we have
 \[
 \#\{1 \leq n_{1}\leq y: n_{1} \equiv l_{1}\textup{ (mod $P_{1}$)}, n_{1}\equiv a\textup{ (mod $n$)}\}
 \ll \frac{y}{P_{1}n} + 1\leq \frac{y/P_{1}}{\varphi (n)}+1.
 \]Applying Lemma \ref{L2:PREP} with $j=-n_{2}$, we obtain
 \[
 \sum_{\substack{1\leq n_{1}\leq y\\n_{1}\equiv l_{1}\textup{ (mod $P_{1}$)}}}E_{12}(n_{1}-n_{2};H)
 \ll \frac{y/P_{1}}{H^{M}} + \textup{exp} (36 \log\log y)
 \ll \frac{y}{P_{1}H^{M}}\ll \frac{y}{P_{1}\log y}.
 \]Therefore
 \[
 \sum_{\substack{1\leq n_{1}, n_{2}\leq y\\ n_{1} \equiv l_{1}\textup{\,(mod $P_{1}$)}\\n_{2}\equiv l_{2}\textup{\,(mod $P_{1}$)}}}E_{12}(n_{1}-n_{2};H)=
 \sum_{\substack{1\leq n_{2}\leq y\\n_{2}\equiv l_{2}\textup{ (mod $P_{1}$)}}}
 \sum_{\substack{1\leq n_{1}\leq y\\n_{1}\equiv l_{1}\textup{ (mod $P_{1}$)}}}E_{12}(n_{1}-n_{2};H)
 \ll \frac{y^{2}}{P_{1}^{2}\log y}.
 \]We see that
 \[
 \sigma_{2}^{2}\sum_{1 \leq l_{1}, l_{2} \leq P_{1}} \mathbb{P}(l_{1}, l_{2} \in \mathbf{S}'_{1})
 \sum_{\substack{1\leq n_{1}, n_{2}\leq y\\ n_{1} \equiv l_{1}\textup{\,(mod $P_{1}$)}\\n_{2}\equiv l_{2}\textup{\,(mod $P_{1}$)}}}E_{12}(n_{1}-n_{2};H)\ll \sigma_{2}^{2}\frac{y^{2}}{P_{1}^{2}\log y} (\sigma_{1}P_{1})^{2}=\frac{(\sigma y)^{2}}{\log y}.
 \]

 The second statement in \eqref{S.2.OSNOVA} is proved similarly. This completes the proof of part \textup{(i)} of Proposition \ref{P2}.

 \section{The proof of part \textup{(ii)} of Proposition \ref{P2}}

Fix $H\in \mathfrak{H}'$. Let us consider the case $j=1$ in \eqref{L1:lambda}, which is
\begin{equation}\label{ii.BASIC.j.1}
\mathbb{E}\sum_{q\in \mathcal{Q}_{H}} \sum_{-(K+1)y< n\leq y}\bm{\lambda}'(H; q, n)= \bigg(1 +
 O\bigg(\frac{\log H}{H^{M-2}}\bigg)\bigg) (K+2)y\# \mathcal{Q}_{H}.
\end{equation}By \eqref{Def_lambda}, the left-hand side expands as
\[
\mathbb{E}\sum_{q\in \mathcal{Q}_{H}} \sum_{-(K+1)y< n\leq y}
\frac{1_{\mathbf{AP}'(KH; q, n)\subset \mathbf{S}'_{2}}}{\sigma_{2}^{\# \mathbf{AP}'(KH; q, n)}}=T.
\]Recall that, according to the definitions \eqref{DECOMP.1} and \eqref{DECOMP.2}, $\mathbf{d}_{1}$ and $\mathbf{d}_{2}$ are independent, and so are $\mathbf{AP}'(KH; q, n)$ and $\mathbf{S}'_{2}$. With $d_{1}$ fixed, $\mathbf{AP}'(KH; q, n)$ is also fixed and we will denote it as $\textup{AP}(KH;q, n)$. We see that
\[
T=\sum_{q\in \mathcal{Q}_{H}} \sum_{-(K+1)y< n\leq y}
\sum_{d_{1}\textup{(mod $P_{1}$)}}
\frac{\mathbb{P}(\mathbf{d}_{1}= d_{1})}{\sigma_{2}^{\# \textup{AP}'(KH; q, n)}} \mathbb{P}\big(\textup{AP}'(KH;q,n)\subset \mathbf{S}'_{2} \big).
\]

Fix $q$, $n$, and $d_{1}$. We are going to apply Lemma \ref{L1.PREP} with $\mathcal{U}=\textup{AP}(KH;q,n)$ and
\[
\mathcal{V}=\bigsqcup_{i=1}^{2}\{n+\alpha_{q,i}+ qh: 1\leq h \leq KH\}.
\] Thus, $l=2[KH]\asymp H$. Since $E_{A}(m;h)$ is increasing in $A$, we obtain
\begin{align*}
&\mathbb{P}\big(\textup{AP}'(KH;q,n)\subset \mathbf{S}'_{2} \big) \\
&= \sigma_{2}^{\#\textup{AP}(KH;q,n)}
\bigg(1+ O\bigg(\frac{1}{H^{M-2}}\bigg)+ O\bigg(\frac{1}{H^{2}}\sum_{\alpha\in I_{q}-I_{q}} \sum_{1\leq h, h' \leq KH}E_{64K^{2}H^{2}} (\alpha+qh-qh';H)\bigg)\bigg).
\end{align*} Thus,
\[
T= \bigg(1+ O\bigg(\frac{1}{H^{M-2}}\bigg)\bigg)(K+2)y\#\mathcal{Q}_{H} + T_{1},
\]where
\[
T_{1}=O\bigg(\frac{y}{H^{2}}\sum_{q\in \mathcal{Q}_{H}}\sum_{\alpha\in I_{q}-I_{q}} \sum_{1\leq h, h' \leq KH}E_{64K^{2}H^{2}} (\alpha+qh-qh';H)\bigg).
\]We observe that $\#(I_{q}-I_{q}) = 3$ and that $|\alpha + qh - qh'| \leq (K+1)y\leq x\log x$ for $\alpha\in I_{q}-I_{q}$ and large $x$. We recall that $E_{A}(m;H)=0$ for $m=0$. Fix $h$ and $h'$. We apply Lemma \ref{L4:PREP} with $w=0$ and $m_{q}^{(i)}=\alpha_{q}^{(i)}+qh-qh'$, $i =1, 2, 3$, where $\alpha_{q}^{(1)} = 0$, $\alpha_{q}^{(2)} = c_{q}$, and $\alpha_{q}^{(3)}=-c_{q}$. We find that
\begin{equation}\label{SUM.E.Q.H}
\sum_{q\in \mathcal{Q}_{H}}\sum_{\alpha\in I_{q}-I_{q}}E_{64K^{2}H^{2}} (\alpha+qh-qh';H)
=\sum_{i=1}^{3} \sum_{q\in \mathcal{Q}_{H}}E_{64K^{2}H^{2}} (m_{q}^{(i)};H)
\ll \frac{\#\mathcal{Q}_{H}\log H}{H^{M-2}}.
\end{equation}Therefore
\[
T_{1}\ll \frac{y}{H^{2}} H^{2} \frac{\#\mathcal{Q}_{H}\log H}{H^{M-2}}=\frac{y\#\mathcal{Q}_{H}\log H}{H^{M-2}}.
\]Thus, \eqref{ii.BASIC.j.1} is proved.

Now we consider the case $j=2$ in \eqref{L1:lambda}, which is
\begin{equation}\label{ii.BASIC.j.2}
\mathbb{E}\sum_{q\in \mathcal{Q}_{H}} \bigg(\sum_{-(K+1)y< n\leq y}\bm{\lambda}'(H; q, n)\bigg)^{2}= \bigg(1 +
 O\bigg(\frac{\log H}{H^{M-2}}\bigg)\bigg) ((K+2)y)^{2}\# \mathcal{Q}_{H}.
\end{equation}The left-hand side is expanded as
\[
\mathbb{E}\sum_{q\in \mathcal{Q}_{H}} \sum_{-(K+1)y< n_{1}, n_{2}\leq y}
\frac{1_{\mathbf{AP}'(KH; q, n_{1})\cup \mathbf{AP}'(KH; q, n_{2})\subset \mathbf{S}'_{2}}}{\sigma_{2}^{\# \mathbf{AP}'(KH; q, n_{1})+ \# \mathbf{AP}'(KH; q, n_{2})}}=T.
\] Arguing as in the case $j=1$, we obtain
\begin{align*}
T&=\sum_{q\in \mathcal{Q}_{H}} \sum_{-(K+1)y< n_{1}, n_{2}\leq y}
\sum_{d_{1}\textup{(mod $P_{1}$)}}
\frac{\mathbb{P}(\mathbf{d}_{1}= d_{1})}{\sigma_{2}^{\# \textup{AP}'(KH; q, n_{1})+ \# \textup{AP}'(KH; q, n_{2})}}\\
&\qquad\qquad\times \mathbb{P}\big(\textup{AP}'(KH;q,n_{1})\cup \textup{AP}'(KH;q,n_{2})\subset \mathbf{S}'_{2} \big)\\
& =\sum_{q\in \mathcal{Q}_{H}} \sum_{-(K+1)y< n_{1}, n_{2}\leq y}
\sum_{d_{1}\textup{(mod $P_{1}$)}} \mathbb{P}(\mathbf{d}_{1}= d_{1}) A(q,n_{1}, n_{2}, d_{1}).
\end{align*}

Fix $q$, $n_{1}$, $n_{2}$, and $d_{1}$. We put
\[
\mathcal{U} = \textup{AP}'(KH;q,n_{1})\cup \textup{AP}'(KH;q,n_{2})
\]and $\mathcal{V} = \mathcal{V}_{1}\cup \mathcal{V}_{2}$, where
\[
\mathcal{V}_{j} = \bigsqcup_{i=1}^{2}\{ n_j+ \alpha_{q, i}+ qh: 1\leq h \leq KH\},\qquad j=1, 2.
\]

We call a triple $(n_{1}, n_{2}, q)$ \emph{good}, if $\mathcal{V}_{1}\cap \mathcal{V}_{2} = \emptyset$, and \emph{bad} otherwise. Suppose that $(n_{1}, n_{2}, q)$ is bad. Then
\[
n_{1}+q h' = n_{2}+ \alpha + q h''
\] for some $\alpha\in I_{q}-I_{q}$ and $1 \leq h', h'' \leq KH$. Therefore, for each $q\in \mathcal{Q}_{H}$, the number of $n_{1}, n_{2}\in (-(K+1)y, y]$ such that $(n_{1}, n_{2}, q)$ is bad is
\begin{equation}\label{BAD.EST}
\ll \sum_{-(K+1)y < n_{2} \leq y} \sum_{\alpha\in I_{q}-I_{q}}
\sum_{\substack{-(K+1)y < n_{1}\leq y\\ n_{1}\equiv n_{2}+ \alpha\textup{\,(mod $q$)}}}1
\ll \sum_{-(K+1)y < n_{2} \leq y} \frac{y}{q} \ll y H,
\end{equation}since $y/q \asymp H$ for $q\in \mathcal{Q}_{H}$.

Using \eqref{H_range} and \eqref{SIGMA.2}, we have
\[
A(q, n_{1}, n_{2}, d_{1}) \leq  \sigma_{2}^{-4KH}
\leq  e^{5K (\log x)^{1/2}}.
\] We obtain
\begin{align*}
\sum_{q\in \mathcal{Q}_{H}} \sum_{\substack{-(K+1)y< n_{1}, n_{2}\leq y\\ (n_{1}, n_{2}, q)\text{\, bad}}}&
\sum_{d_{1}\textup{(mod $P_{1}$)}}\mathbb{P}(\mathbf{d}_{1}= d_{1}) A(q,n_{1}, n_{2}, d_{1})\\
&\ll
yH e^{5K (\log x)^{1/2}} \#\mathcal{Q}_{H}\ll \frac{y^{2}\#\mathcal{Q}_{H}\log H}{H^{M-2}}.
\end{align*}

If $(n_{1}, n_{2}, q)$ is good, then $\textup{AP}'(KH;q,n_{1})$ and  $\textup{AP}'(KH;q,n_{2})$ do not intersect. Therefore
\[
\#\mathcal{U} = \# \textup{AP}'(KH;q,n_{1})+ \# \textup{AP}'(KH;q,n_{2}).
\]Applying Lemma \ref{L1.PREP}, we obtain
\begin{align*}
\sum_{q\in \mathcal{Q}_{H}} \sum_{\substack{-(K+1)y< n_{1}, n_{2}\leq y\\ (n_{1}, n_{2}, q)\text{\, good}}}
&\sum_{d_{1}\textup{(mod $P_{1}$)}}\mathbb{P}(\mathbf{d}_{1}= d_{1}) A(q,n_{1}, n_{2}, d_{1})\\
=& \sum_{q\in \mathcal{Q}_{H}} \sum_{\substack{-(K+1)y< n_{1}, n_{2}\leq y\\ (n_{1}, n_{2}, q)\text{\, good}}} \bigg(1+ O\bigg(\frac{1}{H^{M-2}}\bigg)\bigg)\\
&+O\bigg(\frac{1}{H^{2}}\sum_{q\in \mathcal{Q}_{H}} \sum_{-(K+1)y< n_{1}, n_{2}\leq y}
\sum_{v, v' \in \mathcal{V}} E_{64 K^{2}H^{2}} (v-v'; H)\bigg).
\end{align*}

By \eqref{BAD.EST}, we have
\[
\sum_{q\in \mathcal{Q}_{H}} \sum_{\substack{-(K+1)y< n_{1}, n_{2}\leq y\\ (n_{1}, n_{2}, q)\text{\, bad}}} \bigg(1+ O\bigg(\frac{1}{H^{M-2}}\bigg)\bigg)\ll y H \#\mathcal{Q}_{H} \ll \frac{y^{2}\#\mathcal{Q}_{H}}{H^{M-2}}.
\] We obtain
\begin{align*}
&\sum_{q\in \mathcal{Q}_{H}} \sum_{\substack{-(K+1)y< n_{1}, n_{2}\leq y\\ (n_{1}, n_{2}, q)\text{\, good}}} \bigg(1+ O\bigg(\frac{1}{H^{M-2}}\bigg)\bigg) = \sum_{q\in \mathcal{Q}_{H}} \sum_{-(K+1)y< n_{1}, n_{2}\leq y} \bigg(1+ O\bigg(\frac{1}{H^{M-2}}\bigg)\bigg)\\
 &- \sum_{q\in \mathcal{Q}_{H}} \sum_{\substack{-(K+1)y< n_{1}, n_{2}\leq y\\ (n_{1}, n_{2}, q)\text{\, bad}}} \bigg(1+ O\bigg(\frac{1}{H^{M-2}}\bigg)\bigg) = \bigg(1+ O\bigg(\frac{1}{H^{M-2}}\bigg)\bigg) ((K+2)y)^{2}\#\mathcal{Q}_{H}.
\end{align*}

It is easy to see that
\begin{align*}
&\frac{1}{H^{2}}\sum_{q\in \mathcal{Q}_{H}} \sum_{-(K+1)y< n_{1}, n_{2}\leq y}
\sum_{v, v' \in \mathcal{V}} E_{64 K^{2}H^{2}} (v-v'; H)\\
&\ll \frac{1}{H^{2}}\sum_{q\in \mathcal{Q}_{H}} \sum_{-(K+1)y< n_{1}, n_{2}\leq y}
\sum_{\alpha \in I_{q} - I_{q}} \sum_{1\leq h, h' \leq KH}E_{64 K^{2}H^{2}} (\alpha + q h - qh'; H)\\
&+\frac{1}{H^{2}}\sum_{q\in \mathcal{Q}_{H}} \sum_{-(K+1)y< n_{1}, n_{2}\leq y}
\sum_{\alpha \in I_{q} - I_{q}} \sum_{1\leq h, h' \leq KH}E_{64 K^{2}H^{2}} (n_{1} - n_{2}+\alpha + q h - qh'; H)\\
&= R_{1}+R_{2}.
\end{align*}

Let us estimate $R_{1}$. Fix $n_{1}$, $n_{2}$, $h$, and $h'$. We have $|\alpha+ q h - q h'|\leq (K+1)y\leq x \log x$. We apply Lemma \ref{L4:PREP} with $w=0$ and $m_{q}^{(i)}=\alpha_{q}^{(i)}+qh-qh'$, $i =1, 2, 3$, where $\alpha_{q}^{(1)} = 0$, $\alpha_{q}^{(2)} = c_{q}$, and $\alpha_{q}^{(3)}=-c_{q}$. We obtain
\[
\sum_{q\in \mathcal{Q}_{H}}\sum_{\alpha\in I_{q}-I_{q}}E_{64K^{2}H^{2}} (\alpha+qh-qh';H)
=\sum_{i=1}^{3} \sum_{q\in \mathcal{Q}_{H}}E_{64K^{2}H^{2}} (m_{q}^{(i)};H)
\ll \frac{\#\mathcal{Q}_{H}\log H}{H^{M-2}}.
\]Therefore
\[
R_{1} \ll \frac{\#\mathcal{Q}_{H}\log H}{H^{M-2}} y^{2}.
\]

Let us estimate $R_{2}$. Fix $h$, $h'$, $q$, $\alpha$, and $n_{2}$. We see that
\[
\#\{-(K+1) y < n_{1}\leq y: n_{1}\equiv a\textup{ (mod $n$)}\}\ll \frac{y}{n}+ 1\leq \frac{y}{\varphi (n)}+1
\]for any $n\in \mathcal{D}_{H}\backslash \{1\}$ and any $a\in \mathbb{Z}/ n \mathbb{Z}$. Applying Lemma \ref{L2:PREP} with $\{m_{t}\} = \{ n_{1}: -(K+1)y< n_{1}\leq y\}$, $X = y$, $R=1$, $A=64K^{2}H^{2}$, and $j=-n_{2}+\alpha+ qh-qh'$, we get
\begin{align*}
\sum_{-(K+1)y< n_{1}\leq y}& E_{64 K^{2}H^{2}} (n_{1} - n_{2}+\alpha+ qh - qh'; H)\\
&\ll
y \frac{H^{2}}{H^{M}}+ \textup{exp}(192 K^{2}H^{2}\log\log y)\ll \frac{y}{H^{M-2}}.
\end{align*}Therefore
\[
R_{2} \ll \frac{y^{2}\#\mathcal{Q}_{H}}{H^{M-2}},
\]and \eqref{ii.BASIC.j.2} is proved. The statement \eqref{L1:lambda.2} is proved similarly.

\section{The proof of part \textup{(iii)} of Proposition \ref{P2}}
Now we consider the case $j=1$ in \eqref{L3:lamda.AP.1}, which is
\begin{align*}
\mathbb{E}\sum_{n\in \mathbf{S}'\cap [1, y]} \sum_{q\in \mathcal{Q}_{H}} \sum_{h\leq KH}\bm{\lambda}'(H; q, &n-\alpha_{q, i}-qh)\notag\\
 &= \bigg(1 + O\bigg(\frac{\log H}{H^{M-2}}\bigg)\bigg) \#\mathcal{Q}_{H}[KH]\sigma_{1}y.
\end{align*} It is enough to show that, for any $h\leq KH$,
\begin{align}\label{AP.j.1}
\mathbb{E}\sum_{n\in \mathbf{S}'\cap [1, y]} \sum_{q\in \mathcal{Q}_{H}} \bm{\lambda}'(H; q, &n-\alpha_{q, i}-qh)\notag\\
 &= \bigg(1 + O\bigg(\frac{\log H}{H^{M-2}}\bigg)\bigg) \#\mathcal{Q}_{H}\sigma_{1}y.
\end{align} By \eqref{Def_lambda}, the left-hand side is equal to
\begin{align*}
\mathbb{E}\sum_{n\in \mathbf{S}'\cap [1, y]} \sum_{q\in \mathcal{Q}_{H}} \frac{1_{\mathbf{AP}'(KH; q, n-\alpha_{q,i} - qh)\subset \mathbf{S}'_{2}}}{\sigma_{2}^{\# \mathbf{AP}'(KH; q, n-\alpha_{q,i}-qh)}}=T.
\end{align*}We have
\[
T=\sum_{d_{1}\textup{\,(mod $P_{1})$}}\mathbb{P}(\mathbf{d}_{1}=d_{1})\sum_{d_{2}\textup{\,(mod $P_{2})$}}\mathbb{P}(\mathbf{d}_{2}=d_{2})\sum_{n\in \textup{S}'\cap [1, y]} \sum_{q\in \mathcal{Q}_{H}} \frac{1_{\textup{AP}'(KH; q, n-\alpha_{q,i} - qh)\subset \textup{S}'_{2}}}{\sigma_{2}^{\# \textup{AP}'(KH; q, n-\alpha_{q,i}-qh)}}.
\]

 By \eqref{DECOMP.FINAL}, the condition $n\in \textup{S}' \cap [1,y]$ implies that $n\in \textup{S}'_{1}\cap [1,y]$. On the other hand, if $n\in \textup{S}'_{1}$, then $n\in \textup{AP}'(KH; q, n-\alpha_{q,i}-qh)$, and thus the condition $n\in \textup{S}'_{2}$ is contained in the condition $\textup{AP}'(KH; q, n-\alpha_{q,i}-qh) \subset \textup{S}'_{2}$. For each $d_{1}$, $d_{2}$, $q$, we get
 \[
 \sum_{n\in \textup{S}'\cap [1, y]}  \frac{1_{\textup{AP}'(KH; q, n-\alpha_{q,i} - qh)\subset \textup{S}'_{2}}}{\sigma_{2}^{\# \textup{AP}'(KH; q, n-\alpha_{q,i}-qh)}}=
 \sum_{n\in \textup{S}'_{1}\cap [1, y]}  \frac{1_{\textup{AP}'(KH; q, n-\alpha_{q,i} - qh)\subset \textup{S}'_{2}}}{\sigma_{2}^{\# \textup{AP}'(KH; q, n-\alpha_{q,i}-qh)}}.
 \]We obtain
\begin{align*}
T&=\sum_{d_{1}\textup{\,(mod $P_{1})$}}\mathbb{P}(\mathbf{d}_{1}=d_{1})\sum_{d_{2}\textup{\,(mod $P_{2})$}}\mathbb{P}(\mathbf{d}_{2}=d_{2})\sum_{n\in \textup{S}'_{1}\cap [1, y]} \sum_{q\in \mathcal{Q}_{H}} \frac{1_{\textup{AP}'(KH; q, n-\alpha_{q,i} - qh)\subset \textup{S}'_{2}}}{\sigma_{2}^{\# \textup{AP}'(KH; q, n-\alpha_{q,i}-qh)}}\\
&=\sum_{d_{1}\textup{\,(mod $P_{1})$}}\mathbb{P}(\mathbf{d}_{1}=d_{1})\sum_{n\in \textup{S}'_{1}\cap [1, y]} \sum_{q\in \mathcal{Q}_{H}} \frac{\mathbb{P}(\textup{AP}'(KH; q, n-\alpha_{q,i} - qh)\subset \mathbf{S}'_{2})}{\sigma_{2}^{\# \textup{AP}'(KH; q, n-\alpha_{q,i}-qh)}}.
\end{align*}We apply Lemma \ref{L1.PREP} with $\mathcal{U} = \textup{AP}'(KH; q, n-\alpha_{q,i} - qh)$ and
\[
\mathcal{V}=\bigsqcup_{s=1}^{2}\{n-\alpha_{q,i} - qh +\alpha_{q,s}+ qh': 1\leq h' \leq KH\}.
\]Hence, $l=2[KH]$. We see that
\begin{align*}
&\mathbb{P}(\textup{AP}'(KH; q, n-\alpha_{q,i} - qh)\subset \mathbf{S}'_{2})=\sigma_{2}^{\# \textup{AP}'(KH; q, n-\alpha_{q,i}-qh)}\\
&\times\bigg(1+ O\bigg(\frac{1}{H^{M-2}}\bigg)+O\bigg(\frac{1}{H^{2}}\sum_{\alpha\in I_{q}-I_{q}}\sum_{h', h'' \leq KH} E_{64 K^{2}H^{2}}(\alpha + q h' - q h''; H)\bigg)\bigg).
\end{align*}

 Arguing as in \eqref{SUM.E.Q.H}, we obtain
\[
T=\bigg(1+O\bigg(\frac{\log H}{H^{M-2}}\bigg)\bigg)\#\mathcal{Q}_{H}\mathbb{E}\#(\mathbf{S}'_{1}\cap[1,y])
=\bigg(1+O\bigg(\frac{\log H}{H^{M-2}}\bigg)\bigg)\#\mathcal{Q}_{H}\sigma_{1}y,
\]and \eqref{AP.j.1} is proved. The case $j=1$ in \eqref{L3:lamda.AP.2} is proved similarly.

Now we consider the case $j=2$ in \eqref{L3:lamda.AP.1}, which is
\begin{align}\label{j=2.HOME}
\sum_{h_{1}, h_{2}\leq KH}\mathbb{E}\sum_{n\in \mathbf{S}'\cap [1, y]} \sum_{q_{1},q_{2}\in \mathcal{Q}_{H}}&\bm{\lambda}'(H; q_{1}, n-\alpha_{q_{1}, i}-q_{1}h_{1})\bm{\lambda}'(H; q_{2}, n-\alpha_{q_{2}, i}-q_{2}h_{2})\notag\\
 &= \bigg(1 + O\bigg(\frac{\log H}{H^{M-2}}\bigg)\bigg) (\#\mathcal{Q}_{H}[KH])^{2}\frac{\sigma_{1}}{\sigma_{2}} y.
\end{align}Arguing as above, we see that the left-hand side is equal to
\begin{align}\label{FINAL.T}
&\sum_{h_{1}, h_{2}\leq KH}\sum_{d_{1}\textup{\,(mod $P_{1})$}}\mathbb{P}(\mathbf{d}_{1}=d_{1})\sum_{d_{2}\textup{\,(mod $P_{2})$}}\mathbb{P}(\mathbf{d}_{2}=d_{2})\notag\\
&\times \sum_{n\in \textup{S}'_{1}\cap [1, y]} \sum_{q_{1}, q_{2}\in \mathcal{Q}_{H}} \frac{1_{\textup{AP}'(KH; q_{1}, n-\alpha_{q_{1},i} - q_{1}h_{1})\cup \textup{AP}'(KH; q_{2}, n-\alpha_{q_{2},i} - q_{2}h_{2}) \subset \textup{S}'_{2}}}{\sigma_{2}^{\# \textup{AP}'(KH; q_{1}, n-\alpha_{q_{1},i}-q_{1}h_{1})+\# \textup{AP}'(KH; q_{2}, n-\alpha_{q_{2},i}-q_{2}h_{2}) }}\notag\\
&=\sum_{h_{1}, h_{2}\leq KH}\sum_{d_{1}\textup{\,(mod $P_{1})$}}\mathbb{P}(\mathbf{d}_{1}=d_{1})\sum_{n\in \textup{S}'_{1}\cap [1, y]} \sum_{q_{1}, q_{2}\in \mathcal{Q}_{H}}\notag\\
&  \frac{\mathbb{P}(\textup{AP}'(KH; q_{1}, n-\alpha_{q_{1},i} - q_{1}h_{1})\cup \textup{AP}'(KH; q_{2}, n-\alpha_{q_{2},i} - q_{2}h_{2}) \subset \mathbf{S}'_{2})}{\sigma_{2}^{\# \textup{AP}'(KH; q_{1}, n-\alpha_{q_{1},i}-q_{1}h_{1})+\# \textup{AP}'(KH; q_{2}, n-\alpha_{q_{2},i}-q_{2}h_{2}) }}=T.
\end{align} The contribution from $q_{1}=q_{2}$ is
\[
\ll \sigma_{2}^{-4KH}\#\mathcal{Q}_{H}\sigma_{1}y H^{2}\ll \frac{\log H}{H^{M-2}}(\#\mathcal{Q}_{H}H)^{2}\frac{\sigma_{1}}{\sigma_{2}} y.
\]

For $q\in\mathcal{Q}_{H}$ and $h\leq KH$, we put
\[
\mathcal{V}(q,h)=\bigsqcup_{j=1}^{2}\{\alpha_{q,j}-\alpha_{q,i}+q (h'-h): h' \leq KH\}.
\]We observe that $0\in \mathcal{V}(q,h)$. We call a pair $(q_{1}, q_{2})\in \mathcal{Q}_{H}^{2}$ with $q_{1}\neq q_{2}$ \emph{good}, if for all $h_{1}, h_{2} \leq KH$  we have
\[
\mathcal{V}(q_{1}, h_{1})\cap \mathcal{V}(q_{2}, h_{2}) = \{0\},
\]and call $(q_{1},q_{2})$ \emph{bad} otherwise. If a pair $(q_{1}, q_{2})$ is bad, then there are $1\leq h_{2}, h'_{2}\leq KH$ and $j\in\{1,2\}$ such that
\[
n_{0} = \alpha_{q_{2}, j} - \alpha_{q_{2}, i}+ q_{2}(h'_{2} - h_{2})\neq 0
\]and $n_{0}\textup{(mod $q_{1}$)}\in I_{q_{1}}- I_{q_{1}}$. Thus, $q_{1}$ divides $F(n_{0})=n_{0}(an_{0}+b)(-an_{0}+b)$. Since $n_{0}\neq 0$, we have $F(n_{0})\neq 0$. Hence, the number of such $q_{1}$ is $\ll \log y$. We see that the number of bad pairs $(q_{1},q_{2})$ is
\begin{equation}\label{BAD.PAIRS}
\ll \#\mathcal{Q}_{H}H^{2}\log y.
\end{equation} Hence, the contribution to \eqref{FINAL.T} from bad pairs is
\[
\ll \sigma_{2}^{-4KH}\#\mathcal{Q}_{H}H^{4}(\log y) \sigma_{1}y\ll \frac{\log H}{H^{M-2}}(\#\mathcal{Q}_{H}H)^{2}\frac{\sigma_{1}}{\sigma_{2}} y.
\]

It remains to estimate the contribution to \eqref{FINAL.T} from good pairs. Note that if $(q_{1},q_{2})$ is a good pair, then for any $h_{1}$, $h_{2}$, $d_{1}$, $n$ the set
\[
\textup{AP}'(KH; q_{1}, n-\alpha_{q_{1},i} - q_{1}h_{1})\cup \textup{AP}'(KH; q_{2}, n-\alpha_{q_{2},i} - q_{2}h_{2})
\]has size $\# \textup{AP}'(KH; q_{1}, n-\alpha_{q_{1},i}-q_{1}h_{1})+\# \textup{AP}'(KH; q_{2}, n-\alpha_{q_{2},i}-q_{2}h_{2}) -1$. Applying Lemma \ref{L1.PREP} and \eqref{BAD.PAIRS}, we see that the sum in \eqref{FINAL.T} corresponding to good pairs $(q_{1}, q_{2})$ is equal to
\[
(\#\mathcal{Q}_{H}[KH])^{2}\frac{\sigma_{1}}{\sigma_{2}}y\bigg(1+ O\bigg(\frac{\log H}{H^{M-2}}\bigg)\bigg) +R,
\]where
\begin{equation}\label{R.FINAL}
R=O\bigg(\frac{\sigma_{1}y}{\sigma_{2}H^{2}}R_{0}\bigg)
\end{equation} and
\begin{align*}
R_{0} &= \sum_{h_{1}, h_{2}\leq KH} \sum_{q_{1}, q_{2}\in \mathcal{Q}_{H}}\bigg(\sum_{h'_{1},h'_{2}\leq KH}\sum_{\alpha_{1}\in I_{q_{1}}-I_{q_{1}}} E_{64K^{2}H^{2}}(\alpha_{1}+q_{1}h'_{1}-q_{1}h'_{2};H)\\
&\ \ \ +\sum_{h'_{1},h'_{2}\leq KH}\sum_{\alpha_{2}\in I_{q_{2}}-I_{q_{2}}} E_{64K^{2}H^{2}}(\alpha_{2}+q_{2}h'_{1}-q_{2}h'_{2};H)\\
&\ \ \ +\sum_{h'_{1},h'_{2}\leq KH}\sum_{\alpha_{1}\in I_{q_{1}}-I_{q_{1}}}\sum_{\alpha_{2}\in I_{q_{2}}-I_{q_{2}}}  E_{64K^{2}H^{2}}(\alpha_{1}-\alpha_{2}-q_{1}h_{1}+q_{1}h'_{1}+q_{2}h_{2}-q_{2}h'_{2};H)\bigg)\\
&=R_{1}+R_{2}+R_{3}.
\end{align*}Arguing as in \eqref{SUM.E.Q.H}, we see that
\[
\sum_{q_{1}\in \mathcal{Q}_{H}}\sum_{\alpha_{1}\in I_{q_{1}}-I_{q_{1}}}E_{64K^{2}H^{2}}(\alpha_{1}+q_{1}h'_{1}-q_{1}h'_{2};H)
\ll \frac{\#\mathcal{Q}_{H}\log H}{H^{M-2}}.
\]Therefore
\[
R_{1}\ll \frac{(\#\mathcal{Q}_{H})^{2}H^{4}\log H}{H^{M-2}}.
\]Arguing similarly,  we obtain the same bound for $R_{2}$.

Now we estimate $R_{3}$. Using the fact that $E_{A}(m;H)=0$ for $m=0$, we find that
\[
R_{3}\leq R_{4}+R_{5}+R_{6}+R_{7},
\]where
\begin{align*}
R_{4}&= \sum_{\substack{h_{1}, h_{2}, h'_{1},h'_{2}\leq KH\\h_{1}\neq h'_{1}}} \sum_{q_{1}, q_{2}\in \mathcal{Q}_{H}}
\sum_{\substack{\alpha_{1}\in I_{q_{1}}-I_{q_{1}}\\\alpha_{2}\in I_{q_{2}}-I_{q_{2}}}}  E_{64K^{2}H^{2}}(\alpha_{1}-\alpha_{2}-q_{1}h_{1}+q_{1}h'_{1}+q_{2}h_{2}-q_{2}h'_{2};H),\\
 R_{5}&= \sum_{\substack{h_{1}, h_{2}, h'_{1},h'_{2}\leq KH\\h_{2}\neq h'_{2}}} \sum_{q_{1}, q_{2}\in \mathcal{Q}_{H}}
\sum_{\substack{\alpha_{1}\in I_{q_{1}}-I_{q_{1}}\\\alpha_{2}\in I_{q_{2}}-I_{q_{2}}}}  E_{64K^{2}H^{2}}(\alpha_{1}-\alpha_{2}-q_{1}h_{1}+q_{1}h'_{1}+q_{2}h_{2}-q_{2}h'_{2};H),\\
R_{6}&= \sum_{h,h'\leq KH} \sum_{q_{1}, q_{2}\in \mathcal{Q}_{H}}
\sum_{\alpha_{2}\in I_{q_{2}}-I_{q_{2}}} \sum_{\substack{\alpha_{1}\in I_{q_{1}}-I_{q_{1}}\\ \alpha_{1}\neq 0}} E_{64K^{2}H^{2}}(\alpha_{1}-\alpha_{2};H),\\
R_{7}&= \sum_{h,h'\leq KH} \sum_{q_{1}, q_{2}\in \mathcal{Q}_{H}}
\sum_{\alpha_{1}\in I_{q_{1}}-I_{q_{1}}} \sum_{\substack{\alpha_{2}\in I_{q_{2}}-I_{q_{2}}\\ \alpha_{2}\neq 0}} E_{64K^{2}H^{2}}(\alpha_{1}-\alpha_{2};H).
\end{align*}

 Let us estimate $R_{4}$. Fix $h_{1}$, $h_{2}$, $h'_{1}$, $h'_{2}$, $q_{2}$, and $\alpha_{2}$. We apply Lemma \ref{L4:PREP} with $w=-\alpha_{2}+q_{2}h_{2}-q_{2}h'_{2}$ and $m_{q_{1}}^{(i)}=\alpha_{q_{1}}^{(i)}+q_{1}h'_{1}-q_{1}h_{1}$, $i =1, 2, 3$, where $\alpha_{q_{1}}^{(1)} = 0$, $\alpha_{q_{1}}^{(2)} = c_{q_{1}}$, and $\alpha_{q_{1}}^{(3)}=-c_{q_{1}}$. We find that
\begin{align*}
\sum_{q_{1}\in \mathcal{Q}_{H}}\sum_{\alpha_{1}\in I_{q_{1}}-I_{q_{1}}}&E_{64K^{2}H^{2}}(\alpha_{1}-\alpha_{2}-q_{1}h_{1}+q_{1}h'_{1}+q_{2}h_{2}-q_{2}h'_{2};H)\\
&\ \ \ \ \ =\sum_{i=1}^{3} \sum_{q_{1}\in \mathcal{Q}_{H}}E_{64K^{2}H^{2}} (m_{q_{1}}^{(i)}+w;H)
\ll \frac{\#\mathcal{Q}_{H}\log H}{H^{M-2}}.
\end{align*}We obtain
\[
R_{4}\ll \frac{(\#\mathcal{Q}_{H})^{2}H^{4}\log H}{H^{M-2}}.
\]Arguing similarly, we obtain the same bound for $R_{5}$.

Let us estimate $R_{6}$. Fix $h$, $h'$, $q_{2}$, and $\alpha_{2}$. We apply Lemma \ref{L4:PREP} with $w=-\alpha_{2}$ and $m_{q_{1}}^{(i)}=\alpha_{q_{1}}^{(i)}$, $i =1, 2$, where $\alpha_{q_{1}}^{(1)} = c_{q_{1}}$, $\alpha_{q_{1}}^{(2)} = -c_{q_{1}}$. We find that
\[
\sum_{q_{1}\in \mathcal{Q}_{H}}\sum_{\substack{\alpha_{1}\in I_{q_{1}}-I_{q_{1}}\\\alpha_{1}\neq 0}}E_{64K^{2}H^{2}}(\alpha_{1}-\alpha_{2};H) =\sum_{i=1}^{2} \sum_{q_{1}\in \mathcal{Q}_{H}}E_{64K^{2}H^{2}} (m_{q_{1}}^{(i)}+w;H)
\ll \frac{\#\mathcal{Q}_{H}\log H}{H^{M-2}}.
\]We obtain
\[
R_{6}\ll \frac{(\#\mathcal{Q}_{H})^{2}H^{2}\log H}{H^{M-2}}.
\]Arguing similarly, we obtain the same bound for $R_{7}$.

Gathering estimates for $R_{i}$, $i=1,\ldots, 7$, together,  we obtain
\[
R_{0}\ll \frac{(\#\mathcal{Q}_{H})^{2}H^{4}\log H}{H^{M-2}}.
\]Thus, by \eqref{R.FINAL},
\[
R\ll \frac{\sigma_{1}y}{\sigma_{2}H^{2}}R_{0}\ll \frac{(\#\mathcal{Q}_{H}H)^{2}\log H}{H^{M-2}}\frac{\sigma_{1}y}{\sigma_{2}}.
\] We obtain
\[
T=(\#\mathcal{Q}_{H}[KH])^{2}\frac{\sigma_{1}}{\sigma_{2}}y\bigg(1+ O\bigg(\frac{\log H}{H^{M-2}}\bigg)\bigg)
\]and \eqref{j=2.HOME} is proved. The case $j=2$ in \eqref{L3:lamda.AP.2} is proved similarly. This completes the proof of Proposition \ref{P2}.

\section{Acknowledgements}

This work is an output of a research project (HSE-BR-2025-024) implemented as part of the Basic Research Program at HSE University.

\end{document}